\documentclass[11pt,reqno]{amsart}

\calclayout

\usepackage{amsmath}
\usepackage{amssymb}
\usepackage{amsthm}
\usepackage{comment}
\usepackage{stix2}
\usepackage{xcolor}

\usepackage{url}

\usepackage{upgreek}
\usepackage{bm}

\usepackage[colorlinks = true,
linkcolor = blue,
urlcolor  = blue,
citecolor = violet,
anchorcolor = blue, pdfa]{hyperref}

\usepackage{cleveref}

\newtheorem{theorem}{Theorem}[section]
\newtheorem{lemma}[theorem]{Lemma}
\newtheorem{proposition}[theorem]{Proposition}
\newtheorem{corollary}[theorem]{Corollary}
\newtheorem{definition}[theorem]{Definition}
\newtheorem{example}[theorem]{Example}
\newtheorem{remark}[theorem]{Remark}

\newcommand{\R}{{\boldsymbol{\mathbb R}}}

\newcommand{\N}{{\boldsymbol{\mathbb N}}}
\newcommand{\Z}{{\boldsymbol{\mathbb Z}}}
\newcommand{\vecf}[1]{\mathfrak{X}(#1)}
\newcommand{\vecfp}[1]{\mathfrak{X}^\perp(#1)}
\newcommand{\grad}{\operatorname{grad}}
\newcommand{\divg}{\operatorname{div}}

\newcommand{\inner}[2]{\langle #1, #2\rangle}
\newcommand{\im}{\operatorname{Im}}
\newcommand{\trace}{\operatorname{tr}}
\newcommand{\codim}{\operatorname{codim}}

\usepackage[backend=biber]{biblatex}
\title{On the genericity of singularities in spacetimes with weakly trapped submanifolds}

\author{Ivan P. Costa e Silva}
\address{
	Department of Mathematics, Universidade Federal de Santa Catarina, 88.040-900\\
	Florianópolis-SC, Brazil}
\email{pontual.ivan@ufsc.br}

\author{Victor L. Espinoza}
\address{
	Department of Mathematics, Universidade Federal de Santa Catarina, 88.040-900\\
	Florianópolis-SC, Brazil}
\email{victor.luis.espinoza@gmail.com}


\begin{document}
	
	\begin{abstract}
		We investigate suitable, physically motivated conditions on spacetimes containing certain submanifolds - the so-called \textit{weakly trapped submanifolds} - that ensure, in a set of neighboring metrics with respect to a convenient topology, that the phenomenon of nonspacelike geodesic incompleteness (i.e., the existence of singularities) is \textit{generic} in a precise technical sense. We obtain two sets of results. First, we use strong Whitney topologies on spaces of Lorentzian metrics on a manifold $M$, in the spirit of D.E Lerner's ``The space of Lorentz metrics'' [Commun. Math. Phys. 32, pp. 19–38 (1973)], and obtain that while the set of singular Lorentzian metrics around a fiducial one possessing a weakly trapped submanifold $\Sigma$ is not really generic, it is nevertheless \textit{prevalent} in a sense we define, and thus still quite ``large'' in this sense. We prove versions of that result both for the case when $\Sigma$ has codimension 2, and for the case of higher codimension. 
		The second set of results explore a similar question, but now for initial data sets containing MOTS. For this case, we use certain well-known infinite dimensional, Hilbert manifold structures on the space of initial data and use abstract functional-analytic methods based on the work of Biliotti, Javaloyes,
and Piccione ``Genericity of nondegenerate critical points and Morse geodesic functionals'' [Indiana Univ. Math. J. 58, pp. 1797–1830 (2009)] to obtain a true genericity of null geodesic incompleteness around suitable initial data sets containing MOTS. 
		%		Using the standard Whitney topologies on spaces of Lorentzian metrics, we show that the existence of causal incomplete geodesics is a $C^\infty$-generic feature within the class of spacetimes of a given dimension $n\geq 3$ that are stably causal, satisfy the timelike convergence condition (``strong energy condition'') and contain a codimension-two spacelike weakly trapped closed submanifold such as, e.g., a marginally outer trapped surface (MOTS). By using a singularity theorem of Galloway and Senovilla for spacetimes containing trapped closed submanifolds of codimension higher than two we also prove an analogous $C^\infty$-genericity result for stably causal spacetimes with a suitably modified curvature condition and weakly trapped closed spacelike submanifold of any codimension $k> 2$.
	\end{abstract}
	
	\maketitle
	
	\section{Introduction}
	Let $(M^n,g)$ be \textit{spacetime}, i.e., a connected time-oriented Lorentz manifold of dimension $n\geq2$. An embedded spacelike submanifold (without boundary) $\Sigma\subset M$ is said to be \textit{(future) trapped} (in $(M,g)$) if its mean curvature vector field $H^\Sigma$ is past-directed timelike everywhere on $\Sigma$. (Such submanifolds are also known as a \textit{future-converging}, comp. \cite[Def. 14.57]{OneillSRG}, and the time-dual definitions of \textit{past trapped}/\textit{past-converging} submanifolds are understood throughout even if they will be hardly mentioned.). In particular, compact, codimension-two trapped submanifolds (often called \textit{closed trapped surfaces} in the physics literature) were introduced by Roger Penrose in 1965 in the context of his celebrated first singularity theorem proven in \cite{PenroseTeo}. Penrose's insightful use of trapped submanifolds to model gravitational collapse set up the template for all the other singularity theorems in the literature (those described 
	in \cite{BeemGLG,HawkingTeo,HawkPen} being only the most famous ones). The importance of theorems of this kind in geometric theories of gravity - those in which gravitational fields are described via the geometry of some spacetime irrespective of detailed field equations - cannot be overstated, and indeed the 2020 Nobel prize in physics was shared by Penrose in a direct recognition of this fact. 
	
	It is well known that all singularity theorems establish the existence of incomplete, inextendible causal geodesics in spacetimes under certain physically motivated geometric assumptions, the presence of compact trapped submanifolds being a primary one. Now, their role in these theorems is not as independent physical structures, but rather to provide an abstract, differential-geometric description of the onset of gravitation collapse. As such, trapped submanifolds are widely believed to appear in any physically realistic model where ``large enough concentrations of mass gather in small enough regions''. However, due to their abstract character, a precise mathematical formulation of the informal statement between quotations marks and an actual proof that compact trapped surfaces should indeed develop under suitable conditions are notoriously difficult. Even in the simplest and best-known geometric theory of gravity, general relativity, these proofs require very specific, highly technical hypotheses in order to control the analytic and geometric features of a long-term development of initial data for the Einstein field equations, see e.g., \textcite{christou,KleinRod1,KleinRod2}. (Geometric theories other than general relativity, with different field equations, would have to be analyzed on a case-by-case basis.) Therefore, the existence of suitable trapped submanifolds remains in most cases an added, extra assumption. 
	
	Yet, since the appearance of incomplete causal geodesics are the primary concern in singularity theorems, one might try to find a way to deduce their presence indirectly, without imposing the existence of a trapped submanifold at the outset. For instance, insofar as the interior of black holes are expected to be prime places to find singularities in physics, one might hope to link their presence to the detection of the black horizons themselves. (Indeed, insofar as one cannot ``peek'' behind event horizons, this is arguably the \textit{only} way one can actually do it in physics!). Mathematically, however, black hole horizons have a global definition that seems ill-suited to make them good substitutes for ``quasi-local'' objects like closed trapped surfaces (see, e.g., \cite[pp. 315-323]{HE} and \cite[pp. 27-74]{HBBBH}) in singularity theorems. But an early, general quasi-local description of black hole horizons have been given via \textit{marginally outer trapped surfaces (MOTS)}, which are closely related to (and sometimes conflated with) the so-called \textit{apparent horizons} \cite{HE,WaldGR}. Spatial slices of totally geodesic null hypersurfaces - such as \textit{Killing horizons}, a class that includes the event horizons of stationary black holes in general relativity - give key examples of MOTS. Although in general, dynamic, non-stationary black holes MOTS do not coincide with slices of black hole horizons, they often appear together. Indeed, they are much used in numerical relativity in ``horizon-finder`` algorithms even in a highly dynamic regime far away from stationarity \cite{Lin_2007,thornburg2007event}.
	
	We shall say that an embedded spacelike submanifold (without boundary) $\Sigma\subset M$ is \textit{weakly (future) trapped} (in $(M,g)$) if at each $p \in \Sigma$ its mean curvature vector $H^\Sigma_p$ is either zero or past-directed \textit{causal}. Any trapped submanifold is thus weakly trapped, but so are MOTS. The question now is: can we prove singularity theorems in the presence of weakly trapped submanifolds? Put in this naive way, the answer is of course generally negative: It is easy (via, say, simple isometric identifications of Minkowski spacetime) to find geodesically complete Ricci-flat spacetimes containing weakly trapped submanifolds of any codimension $\geq 1$. However, one intuitively expects that in the presence of weakly trapped submanifolds, ``arbitrarily small perturbations'' of the spacetime metric might produce trapped submanifolds and hence singularities. But it might still occur that such perturbations would have to be very ''special'' to produce trapped manifolds. Since on physical grounds any field can only have its values ascertained within a certain empirical precision, one may change the question and ask \textit{whether the occurrence of singularities is a ``generic'' feature in a suitable set of metrics ``near'' some metric containing MOTS or other weakly trapped submanifolds}. If so, we ``almost always'' should expect singularities to occur, given the large ensemble of black hole horizons believed to occur throughout the visible universe. We only need to make these terms between quotation marks mathematically more precise.
	
	A first hint as to how this can be done is given by a theorem by \textcite[Prop. 1.1]{chrusciel_outer_2014}, in which the \textit{density} of outer trapped surfaces near MOTS on a suitable space of initial data sets in general relativity is established, and since the presence of outer trapped surfaces are enough to establish a singularity theorem in that context, this can be regarded as definite evidence of the genericity mentioned above. In a different vein, there are certain ``generic singularity theorems'' such as those proven by \textcite{Silva_MOTS_2012} which rely on a variant of the so-called \textit{generic condition} on the curvature tensor (cf. \textcite[Section 2.5]{BeemGLG}). This is an assumption already used in the Penrose-Hawking classic singularity theorem (\textcite{HawkPen}), and although the generic condition at first seems to be a somewhat contrived curvature constraint, its ``truly generic'' character has been analyzed at tangent spaces by \textcite{beem_generic_1993}, as well as globally in the master's thesis of \textcite{Larsson_diss}. These results suggest a natural, broader conceptual perspective on the singularity theorems in \textcite{chrusciel_outer_2014,Silva_MOTS_2012}: one can try and obtain density/genericity of a {\it whole class} of singular spacetimes near (with respect to a suitable topology on the space of Lorentzian metrics) a spacetime containing a closed weakly trapped submanifold such as a MOTS.
	
	In order to turn these general musings into definite theorems, we are inspired by two separate approaches. One is the seminal work by \textcite{lerner_space_1973}, who first introduced a natural framework to discuss such stability and genericity issues in mathematically rigorous terms and in full nonlinear generality, namely the \textit{strong Whitney $C^s$ topologies} on the space of Lorentzian metrics on a given manifold. Lerner presented a cogent case for the special suitability of these topologies in mathematical relativity, analyzing and establishing the stability of a number of causal and curvature properties used in the original singularity theorems (see also \textcite[Ch. 7]{BeemGLG} for a detailed discussion and further results and references on the subject). 
	
	Although they are geometrically flexible and conceptually satisfying, Lerner's methods apply to spacetimes as a whole. As opposed to that, an \textit{initial data set} analysis is arguably of greater \textit{physical} importance. In fact, the latter approach seems to be the only viable one in numerical relativity, for example \cite{Cook_2000}, which in turn is the only practical means of producing theoretical templates to be confronted with astrophysical observations. In this context, one wishes to investigate how often the existence of singularities - i.e., incomplete causal geodesics inextendible in the maximal Cauchy development of a given initial data - can be inferred. The natural means here is to use a suitable \textit{Banach manifold structure} of the class of initial data sets on a given manifold \cite{  alias2011manifold, bartnik2005phase, chrusciel2003mapping} and analytic methods quite different from those of Lerner's. The work of \textcite{chrusciel_outer_2014} (see in particular Thm. 1.2 therein) cited above provides a particularly relevant example of this approach. 
	
	% proved the \textit{density} of singularities arising in Cauchy developments of initial data: apart from certain ``exceptional cases'', initial data sets $(g,\mathcal{K})$ on a given manifold $S$, satisfying the dominant energy condition (DEC) and containing a MOTS $\Sigma\subset S$, can be arbitrarily approximated (in the $C^\infty$-topology) by initial data sets - also satisfying DEC - for which $\Sigma$ becomes an {\it outer} trapped surface, and in whose Cauchy development the existence of an incomplete inextendible causal geodesic can be directly proven (if the underlying manifold $S$ is in addition noncompact). 
	
	%The idea of having a dense or generic set satisfying some proprieties tells us how ``big''  or ``frequent'' such proprieties prevail over all Lorentzian metrics, so that, even if there are some not satisfying the given proprieties, they are arbitrarily close to ones that do satisfy these proprieties. %In a more physical sense, knowing the existence of dense or generic metrical proprieties 

	The goal of this paper is to combine the spirit of both \textcite{lerner_space_1973} and \textcite{chrusciel_outer_2014} - though with different technical details - to obtain novel conditions under which causal geodesic incompleteness of spacetimes can be regarded as a generic phenomenon. 
	
	Reflecting these two motivating philosophies, we divide our work into two main parts. Our first set of results use the strong Whitney topologies on the set of Lorentzian metrics on a fixed manifold, in order to analyze how frequently causal geodesic incompleteness occurs in spacetimes containing \textit{weakly trapped} submanifolds. (Recall, these comprise a class that includes MOTS.) Using a concrete example, we show, in particular, that if we consider only those spacetimes with certain physically motivated but strictly enforced geometric conditions, then singularities are \textit{not} generic with respect to Whitney topologies. But upon allowing slight perturbations of these conditions - which seems to be a physically sensible thing to do anyway  - then they are shown to be still \textit{prevalent} in a suitable technical sense we explain. 
	
	On the other hand, in the second part we do obtain a strict genericity result directly influenced by the arguments in \textcite{chrusciel_outer_2014}. We emphasize that whereas these authors obtain \textit{density}, our goal in this section is to obtain \textit{genericity} of incompleteness near MOTS. (To see the difference, consider the standard example of the set of irrational real numbers versus rational ones in $\mathbb{R}$: both subsets are dense therein, but only the former is generic; thus ``almost every number is irrational''.) We also use abstract infinite dimensional (Banach/Hilbert) manifold structures on the set of initial data (under some reasonable restrictions, cf. \textcite{bartnik2005phase} and \textcite{alias2011manifold}), and apply the most basic technical tool to obtain genericity results in this contexts, namely the Sard-Smale theorem (\textcite{smale1965infinite}). 
	
	We view these two sets of results as complementary perspectives on the same problem. As one might expect, there are advantages and disadvantages in each approach. Among the perks of the Whitney topology approach we might count: (i) the proofs are less technical, and rely on already fairly well-known topological techniques, (ii) our curvature assumption in the codimension 2 case (the so-called strong energy condition) is strictly weaker than the dominant energy condition often imposed on initial data sets, (iii) our results require only the presence of weakly trapped submanifolds, a class far larger than just MOTS, (iv) we include a result for higher codimension for little extra cost, and last but not least (v) the causality requirements on spacetime - basically stable causality - are much weaker than those in the second part.  The latter point is relevant especially in physical applications, because unless strong cosmic censorship applies, any incomplete causal geodesic one predicts in the - necessarily globally hyperbolic - maximal Cauchy development of a given initial data set might still be complete on an isometric extension of lower causality. This can be neatly illustrated in the case of initial data induced on a suitable smooth \textit{partial} Cauchy hypersurfaces in anti-de Sitter spacetime, which is geodesically complete, stably causal but not globally hyperbolic. The causal geodesics are all complete in the whole spacetime, but all incomplete in the Cauchy development of the partial Cauchy hypersurface viewed as a spacetime on its own right. 
	
	The initial data approach, on the other hand, is technically more involved, but it is much more convenient to treat the specific case of MOTS, which are after all natural models for (slices of stationary) black hole horizons. Initial data sets have broader applicability in PDE analysis of the Einstein fields equations of general relativity, and also in numerical methods \cite{Baumgarte_Shapiro_2010,Cook_2000}. From a physical perspective, again, one might argue that one can hardly expect to glean actual, direct information of spacetimes as a whole. Rather, all one can expect is to make predictions from current data, and initial data sets are just the model for such a situation. While we can only guarantee genericity of causal incompleteness for the Cauchy development of initial data sets, and in particular only for globally hyperbolic spacetimes, these already cover a vast amount of interesting and relevant cases.

	The rest of the paper is organized as follows. In section \ref{prel} we recall some basic notions and fix notation and conventions, and in subsection \ref{subsecbasic} we briefly review some of the main results in \cite{lerner_space_1973} we shall use. Section \ref{sectmain1} discusses the codimension 2 case, and the respective theorem \ref{mainthm1} is proven in subsection \ref{sect:sing1}. Also, we are able to answer a problem raised by Lerner (\cref{rmk1.1}) that as far we know was left unanswered, and expand on such point (\cref{prop:conterexprops}). Section \ref{sectmain2} discusses analogous results in larger codimension of the weakly trapped submanifold, while theorem \ref{mainthm2} itself is proven in subsection \ref{sect:sing2}. Finally \cref{sectmain3}  is devoted to the initial data/MOTS case. We start with an abstract, Banach manifold genericity method especially adapted from \cite{piccione1}. This abstract approach has the enormous advantage of flexibility: we can choose among a number of variants of Banach/Hilbert manifold structures on initial data sets and on the set of embeddings extant in the literature, subject only to relatively mild technical restrictions. It also bypasses most of the tremendously involved technicalities arising from looking too closely uponthese structures. We then obtain definite genericity results in (\cref{teo:mainMOTS1} and \cref{coro:motsgencoro}) as a relatively straightforward consequence of the abstract machinery.

	%[INCLUIR OS RESULTADOS DA SEGUNDA PARTE]
	
	%\textcite{IvanExtremal}; \textcite{ivanRigidity}, \textcite{GallowayMaximum}
	
	\section{Preliminary Notions}\label{prel}
	
	In this initial section we briefly introduce  terminology and notation that will recur throughout this paper, and also review some topological and geometrical results that will be referenced in our main discussion.

	We shall assume the reader is familiar with the basic concepts of Lorentzian geometry, and refer to the textbooks by \cite{BeemGLG,OneillSRG} for details. 
	
	Throughout this work, $M$ denotes a fixed connected, non-compact smooth (i.e. $C^{\infty}$) real manifold without boundary of dimension $m+k\geq 3.$ 
	\subsection{Genericity and Prevalence} 
	
	f
	As discussed in the Introduction, the goal of this paper is to establish, under suitable conditions, that the appearance of singularities in a spacetime is a \textit{generic} phenomenon. This is made precise via the following standard definition.
	
	\begin{definition}\label{def:topgeneric}
		Let $X$ be a topological space. A subset  $A\subset X$ is said to be \emph{residual} (in $X$) if it contains a countable {intersection} of open dense sets. A property $P$ of elements of $X$ is said to be \emph{generic} if the subset of all those elements of $X$ that possess the property $P$ is residual.
	\end{definition}
	
	The complementary notion of a residual set is called a \emph{meager set}, that is, a set $S\subset X$ contained in a countable union of nowhere dense subsets (sets such that the interior of their closure is empty) of the topological space $X$. If every residual set is dense, we say that the topological space $X$ is a \emph{Baire space}. The well-known Baire category theorem then states that completely metrizable spaces or locally compact Hausdorff spaces are Baire spaces (\textcite{leeTop}, Thm. 4.68). 
	
	While genericity of course implies density on Baire spaces, the converse is certainly false: the set of rational numbers is dense in the real line, but as a countable union of singleton sets it is also a meager set. 
	
	The idea of searching for generic properties is that its complementary meager set is \emph{topologically negligible}: a nowhere dense set is so ``topologically small'' that the interior of its closure is empty, and a meager set is contained in a countable union of such ``small'' sets, which is also viewed as small. While such notion of smallness seems hardly as good as a the measure-theoretical idea of a null set applicable on $\mathbb{R}^n$, and more generally on finite-dimensional manifolds, it is widely adopted in the context we are interested in here. Part of the reason is because there are no infinite-dimensional analogues of the Lebesgue measure. (See  \textcite{Hunt1992PrevalenceAT} for a discussion on this topic and for an alternative notion of smallness in infinite dimensional linear spaces. While closer to measure theory, it is unclear how to adapt their definitions and results to infinite dimensional manifolds, and thus they will not be pursued here. See also \textcite{oxtoby_measure_1980} for a survey of the advantages and shortcomings of Baire category ``smallness'' as opposed to a measure-theoretic one). In any case, the notions of residual sets and genericity, as well as its dual notion of topological smallness are very standard throughout the literature, and have proven to be useful in many geometrical situations. 
	
	We introduce here the following alternative notion of topological smallness that will be relevant for our first set of results.
	
	\begin{definition}\label{def:prevalence}
		Let $X$ be a topological space and let $C\subseteq X$ be a closed set. We shall say that a set $A \subseteq C$ is \emph{prevalent}\footnote{This nomenclature is inspired by the work of \textcite{Hunt1992PrevalenceAT}, but it is unrelated to the measure-theoretic notion of prevalence introduced by these authors.}  in C if $C\setminus A$ is a meager set in $X.$
	\end{definition}
	
	It is clear from this definition that if $A\subset A'\subset C$ and $A$ is prevalent in the closed set $C\subset X$, then $A'$ is also prevalent in $C$. In particular, the (topological) closure in $X$ (or equivalently in $C$) of a prevalent set in $C$ is also prevalent in $C$. Arbitrary unions \textit{and} countable intersections of prevalent subsets of $C$ are also prevalent in $C$. Finally, it should be obvious that the notion of prevalence is nontrivial only when $C$ has itself nonempty interior in $X$, something which will always occur in all our intended applications, and thus will be implicitly assumed from now on. 
	
	The following elementary lemma will be of use later on. 
	
	\begin{lemma}\label{topologicallemma}
		Let $C$ be a closed subset of a topological space $X,$ and let $U \subseteq C$ be an open subset of $X$ such that $int(C) \subseteq \overline{U}.$ Then $C \setminus U$ is nowhere dense in $X$ (and in particular, $U$ is prevalent in $C$).  
	\end{lemma}
	\begin{proof}
		Since $C \setminus  U$ is closed in $X$,  $int\left({\overline{C \setminus U}}\right) = int(C \setminus U).$ Now, $int(C \setminus U) \subseteq int(C) \setminus \overline{U} = \varnothing$ by assumption, implying that $int\left(\overline {C \setminus{U}}\right) = \varnothing,$ and therefore $C \setminus  U$ is nowhere dense in $X.$ 
	\end{proof}
	
	%It is clear that if a closed set $C\subset X$ is itself a Baire space (which will happen, for example, if $X$ is a complete metric space), then any residual set in $C$ is also prevalent in $C.$ But the converse is not true, as the following simple example shows.
	
	As an already mentioned simple concrete example, if we take $C=\mathbb{R}$ with its standard topology, then the set of irrational numbers is prevalent in the real line, whereas the set of rationals is not. More generally, it is evident that when $C=X$ then every residual set is prevalent in $C$. This can be extended to proper closed subsets of the topological space $X$ as follows. 
	
	\begin{proposition}\label{prop:resprev}
		Let $C\subseteq X$ be a closed subset. Then every residual set of $C$ (with respect to the subspace topology on $C$) is prevalent in $C$. 
	\end{proposition}
	\begin{proof}
		Let $A\subseteq C$ be residual in $C.$ Then $A$ contains the intersection $\bigcap_n \mathcal{O}_n$ of a countable collection $\{\mathcal{O}_n\}_{n\in \mathbb{N}}$ of open dense subsets of $C.$ We have $C \setminus A \subseteq \bigcup_{n} C \setminus \mathcal{O}_n,$ with each $ C \setminus \mathcal{O}_n$ being closed in $C,$ therefore closed in $X.$ Now, $int_X (C \setminus \mathcal O_n)$ is an open subset of $C$ contained in $C \setminus \mathcal O_n$ that does not intersect $\mathcal O_n$. Since $\mathcal O_n$ is dense in $C,$ we conclude that $int_X (C \setminus \mathcal O_n)$ must be empty, and $C\setminus A$ is thus a meager set of $X.$ 
	\end{proof}
	
	As the next example shows, the converse to \cref{prop:resprev} is false even if both $X$ and $C$ are Baire spaces.
	%Following the definition of prevalence, we see by \cref{topologicallemma} that an open set $U$ of a topological space $X$ contained in a closed set $C$ of $X$ that is prevalent also has the propriety that $C \setminus \overline U$ is nowhere dense in $X,$ so in our context of topological smallness, prevalence of $U$ in $C$ is a stronger condition than $C\setminus \overline U$ to be nowhere dense. 
	
	\begin{example}\label{ex:referee}
		Let $X = \R^2$ with the usual topology, and $C = \{y\leq 0\} \cup \{x = 0, y\geq 0\}.$ Let $A = \{y \leq 0\}.$ Since 
		%(\cref{fig:1}) 
		$C \setminus A = \{x = 0, y > 0\},$ and the latter line is nowhere dense in $\mathbb{R}^2$, we conclude that $A$ is a closed prevalent subset in $C.$

		However, $C \setminus A = \{x = 0, y > 0\}$ has nonempty interior with respect to the subspace topology (indeed, this whole half line is itself open in $C$). But then $C\setminus A$ cannot be contained in a meager set in $C$, since $C$ is a Baire space and any meager set therein must have empty interior in $C$. We conclude that $A$ cannot be residual in $C$ with respect to the subspace topology.
	\end{example}
	
	In spite of not being in general residual in the closed set $C\subset X$, a prevalent subset $A\subset C$ is still ``large'' in $C$ \textit{relative to the broader space $X$}, and our position here is that this notion of ``largeness'' is still relevant in the context explored in this paper. 
	
	\subsection{ Lorentzian metrics and Whitney topologies}
	
	Let $Sym^2(M)$ be the vector bundle of $(0,2)$-type symmetric tensors on $M$. We denote by $L\subset Sym^2(M)$ the smooth subbundle whose sections are {\it Lorentzian metric tensor fields} (or {\it Lorentzian metrics} for short) on $M$, which is open in $Sym^2(M)$ when the latter is endowed with its standard manifold topology. In other words, a section of $L$ is a map $g$ which associates with each $p\in M$ a symmetric nondegenerate bilinear form $g_p: T_pM\times T_pM \rightarrow \mathbb{R}$ of index 1. Let $\Gamma^r(L)$ be the set of $r$-differentiable sections with $0 \leq r\leq \infty$, that is, the set of Lorentzian metrics $g$ on $M$ whose components $g_{ij}$ in local coordinates have continuous partial derivatives up to order $r$ for $r\geq 1$, or that are simply continuous when $r=0$\footnote{We avoid to refer to these as ``(of class) $C^r$'' to avoid notational confusion with the Whitney topologies discussed ahead.}. By a \textit{metric} we always mean here a Lorentzian metric unless explicitly stated otherwise, but we specify its degree of differentiability as needed.

	Let us briefly describe the relevant topologies we shall adopt on  $\Gamma^r(L)$. Again, since this is fairly standard material we shall only do so in outline here. In general, given any other smooth manifold $N$, we  denote by $C^r(M,N)$ the set of all $r$-differentiable functions $f : M \to N.$ Over this set we shall always adopt the {\it strong $C^s$  Whitney topologies}, $0\leq s\leq r$ (see, e.g., \cite{hirsch_differential_1976} or \cite{mukherjee_differential_2015} for a comprehensive introduction to Whitney topologies, and  \cite{Mather} for some of the more technical results used here). A basis for the (strong) $C^s$ Whitney topology on $C^r(M,N)$ ($s<\infty$) can be most readily defined via jet bundles as follows. Given the bundle $J^s(M,N)$ of $s$-jets of $r$-differentiable maps, for each open set $\mathcal O \subseteq J^s(M,N)$ put 
	\begin{equation*}\label{eq:basisjet}
		B_s(\mathcal{O}) = \{f \in C^r(M,N) :   j^s f(M) \subseteq \mathcal O\},
	\end{equation*}
	so the collection of all sets of this form is the desired basis. A basis for the so-called Whitney $C^\infty$ topology on $C^\infty(M,N)$ is defined by taking as a basis the collection of all $C^s$-open sets, for all $0 \leq s < \infty.$

	In particular, consider a smooth fiber bundle $E$ over $M.$ Then the set $\Gamma^r(E)$ of its $r$-differentiable sections will also said to be given the $C^s$ topology when we take $\Gamma^r(E) \subseteq C^r(M, E)$, adopt on the latter set the $C^s$ Whitney topology, and endow $\Gamma^r(E)$ with the corresponding induced topology.

	An alternative, perhaps more concrete way of defining Whitney topologies is by using local charts. However, this discussion is somewhat lengthy and we refer to the standard literature for the details (conf., e.g., \cite[35]{hirsch_differential_1976}) or \cite[pp. 237-239]{mukherjee_differential_2015}).

	\begin{center}
		\textit{ $\Gamma^r(L)$ will always be assumed to be endowed with the (induced) $C^s$ topology as described. In most of our main results we shall take $s=2\leq r\leq \infty$. In case we fix some $s \in \mathbb{Z}_+$ in connection with the Whitney topology $C^s$ in a statement, that statement is meant to hold separately on $\Gamma^r(L)$ for each $s\leq r \leq \infty$.}
	\end{center}

	In the seminal work \cite{lerner_space_1973} Lerner obtained a number of key results using the strong Whitney topologies on $\Gamma^r(L)$ - some of which are briefly reviewed in the next section - and these have ever since been widely accepted as the natural topologies to adopt in the particular geometric setting of interest here. In any case, we work exclusively with the strong Whitney topologies on $\Gamma^r(L)$ in this paper, and from now on by \emph{$C^s$  topology} we always mean the strong $C^s$ Whitney topology.  
	
	Recall that such topologies are not even first countable if the domain manifold is not compact, so we will rely on net convergence arguments when needed. Also, recall that by a $C^s$ \emph{stable property} in $C^r(M,N),$ [resp. $\Gamma^r(E)$] we mean a property that is valid for all functions in a $C^s$ open subset of $C^r(M,N)$ [resp. $\Gamma^r(E)$]. For each $t > s$ the $C^t$ topology is finer than the $C^s$ topology, so a stable propriety in $C^s$ is also stable in $C^t.$

	\subsection{Previous stability results}\label{subsecbasic}	
	To keep the presentation reasonably self-contained as well as to establish further notation, we reproduce here, for later reference, some classic results established in \cite{lerner_space_1973} that will be relevant for us later on. 
	
	Fix some $0\leq r$ which is either an integer or else $r=\infty$. Denote by $\mathscr{ST}^r\subset \Gamma^r(L)$ the set of $r$-differentiable metrics $g$ such that $(M,g)$ is time-orientable. The first relevant result is the $C^0$-stability of time-orientability. 
	
	\begin{proposition}[\cite{lerner_space_1973}, prop. 4.7]\label{lernerref1}
		Let $X\in \mathfrak{X}(M)$ be an everywhere nonzero vector field. Then the set 
		$$\mathscr{ST}^r(X):= \{g\in \Gamma^r(L) \; : \; \text{$X$ is $g$-timelike}\}$$
		is $C^0$-open in $\Gamma^r(L)$. In particular, $\mathscr{ST}^r$ is $C^0$- open (and therefore $C^s$-open for each $0\leq s\leq r$) in $\Gamma^r(L)$. \hfill  $\Box$
	\end{proposition}
	
	A convenient aspect of working with $\mathscr{ST}^r(X)$ is that we can {\it simultaneously} choose the time-orientation in all of its elements so that $X$ is future-directed, and we shall implicitly assume this choice from now on. Fix a codimension $k$ submanifold $\Sigma \subseteq M$, and denote by $\mathscr{S}_\Sigma^r$ the set of metrics in $\Gamma^r(L)$ for which $\Sigma$ is spacelike. 
	
	\begin{proposition}[\cite{lerner_space_1973}, prop. 4.2]\label{lernerref2}
		$\mathscr{S}_\Sigma^r$ is $C^0$-open in $\Gamma^r(L)$.  \hfill  $\Box$
	\end{proposition}
	Consider also the set $\mathcal{SC}^r\subset \mathscr{ST}^r$ of (time-orientable) metrics that are \emph{stably causal}, meaning that $g \in \mathcal {SC}^r$ if there is a $C^0$-neighborhood $\mathcal{U} \ni g$ in $\Gamma^r(L)$ such that every $g'\in \mathcal{U}$ is a time-orientable causal metric. The set $\mathcal{SC}^r$ is nonempty because $M$ is noncompact \cite[p. 27, item (a)]{lerner_space_1973}. It is $C^0$-open by definition, and if we denote as $\mathcal {CH}^r \subset \mathscr{ST}^r$ the set of chronological, time-orientable metrics, then this set is $C^0$-closed in $\mathscr{ST}^r$, with $\overline{\mathcal {SC}} = \mathcal{CH}$ (closure in the $C^0$ topology)\cite[p. 27, item (b)]{lerner_space_1973}. 
	
	Thus, {\it stably causal metrics are \emph{$C^0$-generic} in the set of chronological metrics, in the sense that they form an open dense set of the latter, or equivalently, that chronological but not stably causal metrics form a nowhere dense subset of the set of all chronological metrics.}
	
	We emphasize, however, that unlike stability, this does {\it not} imply $C^s$-genericity for $s>0$. This is so, of course, because while $\mathcal{SC}^r$ would still be $C^s$-open in $\mathcal {CH}^r$, it might no longer be dense in this finer topology. {\it A consequence for us here is that in order to obtain our higher-order genericity results we shall need to work with stably causal metrics even if the needed singularity theorems only require chronology.}

	Assume now $r\geq 2$ and denote by $\mathcal{SE}^r$ the set of $r$-differentiable metrics $g \in \Gamma^r(L)$ for which the respective Ricci tensor, denoted by $Ric(g),$ satisfies
	\[
	Ric(g)(v,v) >0, \quad v\in TM \text{ $g$-causal}.
	\] 
	This set is $C^2$-open in $\Gamma^r(L)$ \cite[Prop. 4.3]{lerner_space_1973}. Importantly for us here, for each $g\in \mathcal{SE}^r$ all $g$-causal vectors are {\it generic} in $(M,g)$ in the sense of \cite[ch. 2]{BeemGLG}.
	
	Similarly, consider $\mathcal E^r$ the set of $r$-differentiable metrics $g \in \Gamma^r(L)$ satisfying
	\[
	Ric(g)(v,v) \geq 0, \quad v\in TM \text{ $g$-causal}.
	\]
	This set is $C^2$-closed, with $\overline{\mathcal {SE}^r} \subseteq \mathcal E^r$, where now the overbar indicates $C^2$-closure \cite[p. 28, item 4.4]{lerner_space_1973}. The main result  \cite[Prop. 4.5]{lerner_space_1973} we will need in our later arguments is the following relation between $ \mathcal E^r$ and $ \mathcal {SE}^r$:
	
	\begin{theorem}\label{teo:riccinowwhere}
		In the $C^2$ strong Whitney topology, $int(\mathcal E^r) = \mathcal{SE}^r$, for all $2\leq r\leq \infty$. In particular, $\mathcal{S E}^r$ is prevalent in $\mathcal E^r.$
		%$\mathcal E^r\setminus \overline{\mathcal{SE}^r}$ is $C^2$-nowhere dense in $\Gamma^r(L).$   
		\hfill $\ensuremath{\Box}$
	\end{theorem}		
	
	(The first sentence in theorem \ref{teo:riccinowwhere} is actually the statement as proved by Lerner, but the second sense follows immediately by lemma \ref{topologicallemma}.)
	
	\begin{remark}\label{rmk1}
		{\em The set $\mathcal E^r$ consists precisely of those metrics $g$ satisfying the so-called {\it timelike convergence condition}: $Ric(g)(v,v)\geq 0, \forall v\in TM$ timelike. (The inequality also applies to null vectors via limits). This is often referred to as the {\it strong energy condition} in the physics literature, because it arises via the Einstein field equation in the context of general relativity, by coupling the spacetime metric with physically relevant classical matter fields, most of which satisfy it. Hence, it is a very common assumption in singularity theorems. Theorem \ref{teo:riccinowwhere} has thus a simple but very suggestive meaning: only those metrics in a ``negligibly small'' subset of the metrics satisfying the timelike convergence condition do not admit of an arbitrarily close approximation by a metric in $\mathcal{SE}^r$, which is often causal geodesic incomplete.} 
	\end{remark}
	
	\subsection{Initial Data and MOTS}\label{sec:motsandID}
	Since we shall focus on initial data sets that have a MOTS on the last part of the paper (section \ref{sectmain3}), we now review these notions and some main results related to it. 
	\begin{definition}[Initial Data Set]\label{def:InitialDataSet}
		An \emph{initial data set} is a triple $(\mathcal S^n, h, \mathcal{K})$ where $(\mathcal S^n, h)$ is a Riemannian manifold and $\mathcal{K}$ is a symmetric $(0, 2)$-tensor field on $\mathcal S$. Given an initial data $(\mathcal S, h, \mathcal{K})$, we define a function $\rho \in C^{\infty}(\mathcal S)$ and a one-form $J \in \Omega^1(\mathcal S)$, called respectively the \emph{energy density} and \emph{energy-momentum current} associated with the data by
		\begin{equation}
			\begin{aligned}\label{constreq}
				\rho &= \frac{1}{2} \left[	\mathrm{Scal}_h-|\mathcal{K}|_h^2+(\trace_h \mathcal{K})^2 \right], \\
				J &= \divg_h \mathcal{K} - d(\trace_h \mathcal{K}). 
			\end{aligned}
		\end{equation}
	\end{definition}

	The seminal work of \textcite{choquet-bruhat_global_1969} shows that if $(\mathcal{S},h,\mathcal{K})$ a \textit{vacuum initial data set}, i.e., one for which $\rho =0$ and $J=0$, then the manifold $\mathcal{S}$ can be viewed as a suitably embedded Cauchy spacelike huypersurfaces in a uniquely defined maximal globally hyperbolic Ricci-flat spacetime. We summarize this result in the following theorem:

	%[\cite{ChoquetBruhat1952, ChoquetBruhat1969, Lee2019-xc}]
	\begin{theorem} \label{theo:CB_52_69}
		Let $(\mathcal S^n, h)$ be a vacuum data set. Then there exists an $(n+1)$-dimensional Ricci-flat spacetime $(M^{n+1}, g)$ such that $(\mathcal S^n, h)$ isometrically embeds into $(M, g)$ as a Cauchy hypersurface with second fundamental form $\mathcal{K}$. Furthermore, there is a unique (up to isometry) \emph{maximal} such spacetime, in the sense that any other spacetime  satisfying these conditions can be isometrically embedded therein.  \hfill $\ensuremath{\Box}$
	\end{theorem}

	For our applications we are interested in MOTS and its null expansion scalar from an initial data set point of view. Let $\Sigma^{n-1} \subseteq \mathcal S^{n}$ be an embedded submanifold,
	%\footnote{For simplicity we work with an embbeded submanifold for the last of the chapter. Slight modifications to the formulas exhibited here are needed for general immersions, and we return to them on appendix \ref{appendix:mots}.}
	and recall that $\Sigma$ is \emph{two-sided} if it has a trivial normal bundle in $\mathcal S.$   
	
	%Recall that if $\psi : \Sigma \to \mathcal S$ is a codimension 1 immersion with $\mathcal S$ a Riemannian manifold, $\psi$ is said to be \emph{two-sided} if it has a normal bundle in $S.$ 
	
	\begin{definition}[Null expansion and MOTS - Initial data version]
		Let $(\mathcal S, h, \mathcal{K})$ be an initial data set and let $\Sigma \subseteq \mathcal S$ be a two-sided embedded codimension one submanifold, with one of the two normal vectors $\bm \upnu\in \mathfrak{X}^\perp(\Sigma)$ fixed, and referred to as the \emph{outward-pointing unit normal} vector field on $\Sigma.$ The outward \emph{null expansion} $\theta_+$ [resp. inward \emph{null expansion} $\theta_-$] of $\Sigma$ is defined as
		\begin{equation}\label{eq:NullExpansionInitdata}
			\theta_\pm = \trace_\Sigma \mathcal{K} \pm H_{\bm{\upnu}},
		\end{equation}
		where $H_{\bm \upnu}$ is the mean curvature scalar of $\Sigma$ with respect to the normal ${\bm{\upnu}}$ and $\mathcal{K}$ is the second fundamental form of $\Sigma$ with respect to the normal ${\bm{\upnu}}$ and the (partial) trace on $\Sigma$ is taken with respect to the induced metric. Then, $\Sigma$ is said to be
		\begin{itemize}
			\item outer trapped if $\theta_{+} < 0,$
			\item weakly outer trapped if $\theta_{+} \leq 0,$
			\item marginally outer trapped (MOTS) if $\theta_{+} = 0.$
		\end{itemize}
	\end{definition}
	As mentioned in the Introduction, MOTS are important in general relativity because they provide a ``quasi-local'' analogue of the fully global notion of \textit{(black hole) event horizon}, and as such can be adapted to numerical studies, for example \textcite{Cook_2000}. In purely mathematical terms, however, MOTS are important because they provide spacetime/initial data analogues of minimal surfaces (cf. \textcite{danlee}, section 7.5). Indeed, any Riemannian manifold $(\mathcal S, h)$ can be viewed as an initial data set with $ \mathcal{K}\equiv 0$, and MOTS therein are precisely the minimal surfaces. However, while minimal surfaces can be described via a variational formulation, there is no known such formulation for general MOTS.

	\subsubsection{MOTS stability operator}
	A powerful tool for studying minimal surfaces is the notion of \textit{stability}, which comes from the sign of the second variation of the volume measure; this concept of stability can be generalized to the setting of MOTS through the linearization of the null expansion $\theta_+.$ Such notion was introduced by \textcite{AnderssonMOTSStab} (see also \textcite{eduardothesis} for an in- depth review of the subject). For now, the relevant definition is the following. 
	
	\begin{definition}[MOTS stability operator - Initial data]	\label{def:OpStability}
		Let $\Sigma^{n-1}$ be a closed (that is, compact without boundary) MOTS within an initial data set $(\mathcal S^n, h, \mathcal{K})$. We define its \emph{stability operator} $L:C^\infty(\Sigma) \rightarrow C^\infty(\Sigma)$ as the elliptic differential operator given by
		\begin{equation}
			L(\psi) =  - \Delta \psi + 2 \inner{X} {\grad \psi} + (Q  
			+ \operatorname{div} X   - \| X \|^2 ) \psi, \quad \forall \psi \in C^\infty(\Sigma),
		\end{equation}
		where
		\begin{equation}\label{eq:Qdefinition_INITDATA}
			Q = \frac{1}{2} \mathrm{Scal}_\Sigma -[J(\bm{\upnu})+\rho] -  \frac{1}{2}| \mathcal{K}_{\bm{\upnu}} + \mathcal{K}|^2.
		\end{equation}
		Here, all the geometric objects are defined on $\Sigma$, $\bm{\upnu}$ is the outward pointing unit normal vector field on $\Sigma,$ $\mathcal K_{\bm{\upnu}}$ is the scalar second fundamental form of $\Sigma$ with respect to the induced metric from $(\mathcal S, h)$ on the direction $\bm{\upnu},$  $X$ is a vector field, the metric dual to the one-form $\mathcal{K}(\bm{\upnu}, \cdot)$ on $\Sigma$. Finally, $\rho$ and $J$ are defined as in \cref{def:InitialDataSet}.
	\end{definition}

	In the case of time-symmetric initial data ($ \mathcal K = 0$), the operator $L$ reduces to the self-adjoint, classic stability (or Jacobi) operator of minimal surface theory, which basically gives the second variation of the volume. Although the MOTS stability operator $L$ is not self-adjoint in general, it still possesses certain key spectral properties summarized in the next proposition.

	\begin{proposition}[\textcite{AnderssonMOTSStab}, \textcite{Galloway2018}]\label{prop:stabmots}
		Let $\Sigma$ be a closed MOTS within an initial data set $(\mathcal S^n, h, \mathcal{K})$. The following statements hold for the MOTS stability operator $L$.
		\begin{enumerate}
			\item There is a real eigenvalue $\lambda_1=\lambda_1(L)$, called the \emph{principal eigenvalue} of $L$, such that for any other eigenvalue $\mu$, $Re(\mu)\geq \lambda_1$. The associated eigenfunction $\phi \in C^{\infty}(\Sigma)$, $L\phi = \lambda_1 \phi$, is unique up to a multiplicative constant, and can be chosen to be strictly positive. 
			\item $\lambda_1 \geq 0$ (resp., $\lambda_1>0$) if only and if there exist some $\psi \in C^\infty(\Sigma), \psi>0$, such that $L(\psi) \geq 0$ (resp., $L(\psi)>0$).  \hfill $\ensuremath{\Box}$
			
		\end{enumerate}
		\label{lemma:PrincpEigenValueOpL}
	\end{proposition}

	\subsection{A singularity theorem in the presence of a MOTS}\label{sec:motssing}

	Let $\Sigma^{n-1} \subseteq \mathcal{S}^n$ be connected manifolds. We say that $\Sigma$ \emph{separates} $\mathcal S$ if $\mathcal S \setminus \Sigma$ is not connected. In this case, we write $\mathcal S \setminus \Sigma = \mathcal S_+ \sqcup \mathcal S_-,$ where $\mathcal S_{\pm}$ are open submanifolds of $\mathcal S.$  We now give a singularity theorem for spacetimes containing MOTS that will be of crucial importance for our discussion in the second and final part of this paper.

	\begin{theorem}\label{teo:motssingularity}
		Let $(M^{n+1},g)$ be a globally hyperbolic spacetime. Assume the following conditions. 
		\begin{itemize}
			\item[1)] The \emph{lightlike convergence condition} holds on $(M,g)$, i.e., $Ric(v,v)\geq 0, \forall v\in TM$ lightlike.
			\item[2)] There exists a spacelike Cauchy hypersurface $\mathcal S^{n}\subseteq M$ and a connected, compact MOTS $\Sigma^{n-1} \subseteq \mathcal S$ without boundary that separates $\mathcal S,$ with $\mathcal S \setminus \Sigma = S_+ \sqcup S_-,$ where $S_\pm$ are open disjoint sets in $S$ and $\overline{S_+}$ is not compact.
			\item[3)] The principal eigenvalue for the stability operator of $\Sigma$ is not zero.
		\end{itemize}
		Then, $(M,g)$ is null geodesically incomplete.
	\end{theorem}
	
	\begin{proof}(Sketch)
		Our hypotheses are a restriction of a more general study for singularities in the presence of MOTS done by \textcite{Silva_MOTS_2012}, so we just give here an outline of the argument.
		
		Following \cite{Silva_MOTS_2012}, propositions 2.1 and 2.4, one shows using $(2)$, and even when $\Sigma$ is not a MOTS, the existence of a future-inextendible null $\Sigma$-ray $\eta : [0,a) \to M$ with $\eta'(0)$ parallel at the point $\eta(0)$ to $\ell_+ \in \vecfp{\Sigma}$, which is the unique future-pointing normal lightlike vector field that projects on $\Sigma$ onto the outward unit vector field defining $\theta_+$. 
		
		We assume in $(3)$ that the stability operator of $\Sigma$ has a non-zero principal eigenvalue $\lambda_1 \neq 0$, and since by \cref{prop:stabmots} its eigenfunction $\phi$ can be chosen to be strictly positive. Consider the normal variation $t\mapsto \Sigma_t$ of $\Sigma$ in $\mathcal S$ defined by with vector field $V  = \phi \upnu,$ where $\bm \upnu$ is the  outward pointing unit normal vector field of $\Sigma$ in $\mathcal S.$ One can show that 
		\[
		\left. \frac{\partial \theta_+}{\partial t} \right|_{t = 0} = L(\phi) = \lambda_1 \phi;
		\]
		therefore, under this variation we can deform $\Sigma$ to be outer-trapped  (i.e. $\theta_+<0$) by moving $\Sigma$ along the variation either outwards or inwards according to the sign of $\lambda.$ Still denoting this displaced manifold as $\Sigma,$ we then have an outward-pointing null $\Sigma$-ray $\eta : [0,a) \to M$ normal to this outer-trapped $\Sigma$ as argued above, and this null geodesic if future-incomplete because if it were complete, then by using $(1)$ and prop. 10.43 in \textcite{OneillSRG}, there would exist a focal point of $\Sigma$ along $\eta$ which is a contradiction since $\eta$ is a $\Sigma$-ray and hence globally maximal.
	\end{proof}
	
	\begin{remark} {\em 
			The assumption that $\mathcal S$ is separated by $\Sigma$ is not essential in the previous theorem, because even if that is not the case one can show (the proof is analogous to that of prop. 14.48 of \textcite{OneillSRG}) that there exists a smooth covering manifold of $M$ for which there exists a spacelike Cauchy hypersurface $\tilde{\mathcal{S}}$ covering $\mathcal{S}$ and containing a isometric copy of $\Sigma$ separating $\tilde{\mathcal{S}}$. Now, insofar as we are interested in geodesic incompleteness, one may as well work in covering manifolds. However, the non-compactness of the outer region $\overline{S_+}$ must still be enforced in that covering {(so that one can even allow for compactness in the base if the outer region of the covering is non-compact)}. Otherwise, one may consider the following counterexample. Let $(\mathbb{S}^n,\omega_n)$ be the round $n$-sphere with its standard round metric $\omega_n$, and let $(M:=\mathbb{R}\times \mathbb{S}^n, g:= -dt^2 + \omega_n)$ be the $(n+1)$-dimensional Einstein cylinder. One directly checks that $(M,g)$ is geodesically complete and satisfies the lightlike convergence condition. Moreover, $\mathcal S:= \{0\}\times \mathbb{S}^n$ is a totally geodesically spacelike Cauchy hypersurface, so any minimal sphere $\mathbb{S}^{n-1}\subset \mathbb{S}^n$ defines a two-sided separating MOTS $\Sigma \subset \mathcal S$ in the obvious manner. However, for any such the stability operator is shown by direct computation to reduce on $\Sigma$ to
			\begin{equation}\label{eqtointegrate}
				L(\phi)= - \Delta \phi +c\cdot \phi
			\end{equation}
			with $c>0$ is some positive constant. Taking $\phi$ to be a strictly positive principal eingenfunction, integration of \eqref{eqtointegrate} over $\Sigma$ allows one to conclude that $\lambda_1=c$. Therefore, this example satisfies all conditions in theorem \ref{teo:motssingularity} except the non-compactness of the Cauchy hypersurface, and the conclusion fails.}
	\end{remark}

	\section{Weakly trapped submanifolds and prevalence of singularities I: codimension two}\label{sectmain1}
	
	We fix for both this section and the next one a smooth embedded codimension $k\geq 2$ submanifold $\Sigma^m \subseteq M^{m+k}$. In some of our main results we shall need that $\Sigma$ be compact and without boundary, and then we will simply say that $\Sigma$ is \emph{closed}, and let the context distinguish this usage from the standard topological closure. In order to simplify the notation, unless stated otherwise we work in $\Gamma^2(L)$ endowed with the $C^2$ Whitney topology, with the understanding the everything remains valid for $r$-differentiable metrics with $r\geq 2$, and thus omit any $r$ superscripts in what follows. %In what follows we will restrict our ambient topological space to the open set $\mathscr{ST} \cap \mathscr{S}_\Sigma.$  
	
	Let $g \in \Gamma^2(L)$, and let $\nabla=\nabla^g$ denote its Levi-Civita connection. Recall that if $\Sigma$ is spacelike in $(M,g)$, then we can define the second fundamental form tensor (or {\it shape tensor} for short) of $\Sigma$ by 
	\[
	II^\Sigma(g)(V,W):=(\nabla_VW)^\perp, \quad \forall V,W\in \mathfrak{X}(\Sigma).
	\]
	where $\perp$ denotes the normal part with respect to $g$. Given any local $g$-orthonormal frame $\{E_1,\ldots,E_m\}\subset \mathfrak{X}(\Sigma)$ on $\Sigma$, the associated {\it mean curvature vector field of $\Sigma$} is
	\[
	H^\Sigma(g) := tr_\Sigma\, II^\Sigma(g) = \sum_{i=1}^m II^\Sigma(g)(E_i,E_i).
	\] 
	
	When the underlying metric is unambiguously understood, we shall denote the associated mean curvature vector field simply by $H^\Sigma.$ Straightforward coordinate computations and net convergence arguments similar to the ones given in the proof
	%\footnote{The restriction application of a function $f$ defined over $M$ to an  embedded submanifold $\Sigma$ is continous in $C^r$ topology if and only if $\Sigma$ is a proper subset of $M$ (\cite{hirsch_differential_1976}, exercise 10 of pgs. 64-65), and a smooth embedded submanifold is proper if and only if is a closed subset (\cite{lee_smooth_2013}, proposition 5.5 ). }
	of \textcite[Prop. 4.7(b)]{lerner_space_1973} show that the mapping 
	\begin{equation}\label{eq:MCVcontinuity}
		g \in \Gamma^r(L) \mapsto H^\Sigma(g) \in \Gamma^0(TM|_\Sigma)
	\end{equation}
	is continuous in the $C^r$ topology on $\Gamma^r(L)$ for all $1\leq r\leq \infty$, where $TM|_\Sigma$ is the restriction of the tangent bundle $TM$ to $\Sigma$.

	Fix a nowhere-vanishing vector field $X\in \mathfrak{X}(M)$. We now proceed to define the topological space of metrics we shall consider in this section.
	
	Consider the set $\mathscr{ST}(X) \cap \mathscr{S}_\Sigma \cap \mathcal {SC}\subset \Gamma^r(L)$ of $C^2$ stably causal, time-oriented metrics for which $X$ is future-directed timelike, and for which $\Sigma$ is spacelike. Recall this set $C^2$-open in $\Gamma(L)$ conf. props. \ref{lernerref1} and \ref{lernerref2}). In that set, consider the subset $\mathcal A$ of metrics for which $\Sigma$ is also a future-trapped submanifold. Recall that $\Sigma$ is {\it future-trapped} if and only if $H^\Sigma_p$ is past-directed timelike for each $p\in \Sigma$. 
	
	More precisely, 
	\[ 
	\mathcal A = \{g \in \mathscr {ST} (X)\cap \mathscr{S}_\Sigma \cap \mathcal {SC} : 	g(H^\Sigma, H^\Sigma) < 0, g(H^\Sigma, X) > 0 \}.
	\]
	Since $H^\Sigma$ only depends on the metric coefficients and their first derivatives in suitable coordinates (conf. e.g. \textcite[Ex. 1, p. 123] {OneillSRG}) the definition of $\mathcal A$ readily implies (again by arguments entirely analogous to those in the proof of \textcite[Prop. 4.7(b)]{lerner_space_1973}) that this set is $C^1$-open in $\Gamma^2(L)$, and hence $C^2$-open therein as well. Consider also the set
	\[
	\mathcal{FA} = \{g \in \mathscr {ST} (X)\cap \mathscr{S}_\Sigma \cap \mathcal {SC} : 	g(H^\Sigma, H^\Sigma) \leq 0,  g(H^\Sigma, X) \geq 0 \}.
	\]
	
	Observe that for $g \in \mathcal{FA}$ and a point $ p\in \Sigma,$ either $H_p^\Sigma(g)$ is past-directed and causal, or else it is zero. Therefore, $\Sigma$ is weakly future-trapped for a metric $g \in \mathscr {ST} (X)\cap \mathscr{S}_\Sigma \cap \mathcal {SC}$ if and only if $g \in \mathcal {FA}.$ In the context of a fixed $\Sigma$ as we have here, we informally refer to the metrics in $\mathcal A$ themselves as ``future-trapped metrics,'' and analogously we say the metrics in $\mathcal{FA}$ are ``weakly future-trapped''. 
	
	In particular, $\mathcal{FA}$ includes metrics $g$ for which $\Sigma$ is an {\it extremal submanifold}, i.e., $H^\Sigma(g)\equiv 0$ identically, which occurs for example if $\Sigma$ is totally geodesic with respect to $g$. 
	\begin{remark}\label{motsrmk}
		{\em For another key class of examples of elements of $\mathcal{FA}$, assume for the moment that $\Sigma$ is closed and with codimension $k=2$, and that the normal bundle of $\Sigma$ is trivial. (Of course, this can be ensured independently of the choice of the ambient metric provided suitable orientability assumptions on $\Sigma$ are made.) Then, given $g \in \mathscr {ST} (X)\cap \mathscr{S}_\Sigma \cap \mathcal {SC}$ we may choose two normal future-directed null vector fields $\ell_{\pm} \in \mathfrak{X}^\perp(\Sigma)$ globally defined on $\Sigma$ and spanning its normal bundle.  The {\it null expansion scalars} $\theta_\pm \in C^1(\Sigma)$ associated with this choice are defined by
			\[
			\theta_\pm := -g(H^\Sigma(g),K_{\pm}).
			\]
			Observe that in this case $g\in \mathcal{A}$, i.e., it is future-trapped, if and only if $\theta_\pm <0$. If by convention we say that $\ell_+$ is {\it outward-pointing}, then $\Sigma$ is a {\it marginally outer trapped surface} (MOTS\footnote{MOTS are being redefined here just because we had previously considered them only on initial data sets. The relationship between the two definitions is elementary but the details are not important for us here, so we omit them.}) if $\theta_+\equiv 0$ on $\Sigma$. Then, for each $p\in \Sigma$, {since} $H^\Sigma_p,K_+(p)\in (T_p\Sigma)^\perp$ and the latter vector space is a two-dimensional Lorentz space, they can be orthogonal if and only if either $H^{\Sigma}_p$ is zero or null (parallel to $K_+$), and hence $g\in \mathcal{FA}$. In other words, $C^2$ metrics for which $\Sigma$ is a MOTS are all in $\mathcal{FA}$, i.e., MOTS are weakly future-trapped.}
	\end{remark}
	
	\medskip
	Returning now to our main discussion, we evidently have $\mathcal A \subseteq \mathcal {FA}.$ We will be able to prove much more regarding $\mathcal A$ and $\mathcal{FA},$ the main point being that $\mathcal A$ is prevalent in $\mathcal {FA}.$ Informally, {\it weakly future-trapped spacetime metrics can ``almost always'' be arbitrarily $C^2$-approximated by future-trapped ones.} 
	
	A key tool for these proofs will be conformal perturbations of the metric, so let us briefly recall how the mean curvature vector of $\Sigma$ transforms under a conformal change of the metric. Let $g\in \Gamma^2(L)$ and consider the usual conformal transformation $\widehat g = e^{2f} g,$ where $f : M \to \R$ is any smooth function.  The Levi-Civita connection transforms as (conf. \textcite{leeRiemann}, p. 217)
	\[
	\widehat \nabla_X Y = \nabla_X Y + (Xf)Y + (Yf)X - g(X,Y)\, grad_g f, 
	\]
	for $X, Y$ smooth vector fields over $M.$ Then, the shape tensor of $\Sigma$ associated with the metric $\widehat g$  is
	\[
	\widehat{II}(V,W) = 	(\widehat \nabla_V W)^\perp  = II(V,W)  - g(V,W)\, (grad_g f)^\perp,
	\]
	for $V, W$ smooth vector fields tangent to $\Sigma.$ It now readily follows that 
	\begin{equation}\label{eq:confscale2}
		\widehat H^\Sigma = e^{-2f} H^\Sigma - me^{-2f} (grad_g f)^\perp,
	\end{equation}
	and the scalar product of $H^\Sigma$ transforms as
	\begin{equation}\label{eq:MCVconformalscale}
		\begin{aligned}
			\widehat g(\widehat H^ \Sigma, \widehat H^ \Sigma) = {}& e^ {-2f}g(H^\Sigma, H^\Sigma) - 2e^{-2f}m g(H^\Sigma, grad_g f)\\
			& + e^{-2f}m^2g( (grad_g f)^\perp, (grad_g f)^\perp).
		\end{aligned}
	\end{equation}
	
	We are now ready to prove the announced result. 
	
	\begin{theorem}\label{teo:FAnowhere}
		$\mathcal{FA}$ is $C^2$-closed in $ \mathscr {ST}(X) \cap \mathscr{S}_\Sigma \cap \mathcal {SC},$ with $\mathcal{FA} = \overline{\mathcal A},$ and we also have $int (\mathcal{FA}) = \mathcal A.$ In particular, by lemma \ref{topologicallemma} ${\mathcal A}$ is prevalent in $\mathcal {FA}.$
	\end{theorem}
	\begin{proof}
		First, we show that $\mathcal{FA}$ is $C^2$-closed, thus concluding that $\overline{\mathcal A} \subseteq \mathcal {FA}.$ Consider a net of metrics $\{g_\lambda\}_{\lambda \in \Lambda}$ in $\mathcal {FA}$ converging to a metric $g \in  \mathscr {ST}(X) \cap \mathscr{S}_\Sigma \cap \mathcal {SC}$ in the $C^2$ topology.  Denoting by $H_\lambda ^\Sigma$  the mean curvature vector of $\Sigma$ associated with the metric $g_\lambda,$ by the continuity in (\ref{eq:MCVcontinuity}) we have  $H^\Sigma_\lambda \to H^\Sigma$ in the $C^0$ topology on $\Gamma(TM|_\Sigma)$, and since $C^2$ convergence implies $C^0$ convergence, we have a pointwise convergence $g_\lambda(H^\Sigma_\lambda, H^\Sigma_\lambda) \to g(H^\Sigma, H^\Sigma),$ which implies $ g(H^\Sigma, H^\Sigma) \leq 0.$ 		
		Similarly, we have a pointwise convergence $g_\lambda(X, H^\Sigma_\lambda) \to g(X, H^\Sigma),$ thus showing  that $g(X, H^\Sigma) \geq 0,$ i.e., $g \in \mathcal {FA},$ which is therefore closed as claimed.
		
		In order to show that $\mathcal {FA} \subseteq \overline{\mathcal{A}},$ 	let $g \in \mathcal{FA}$; if we can find a sequence $g_n \in \mathcal A$ of metrics converging to $g$ in the $C^2$ topology to $g$ such that 
		\begin{equation}\label{eq:FA=Aset}
			g_n(H_n, H_n) < 0 \quad \text{and} \quad g_n(H_n, X) > 0,
		\end{equation}
		then we may conclude that $g \in \overline{\mathcal A}.$ 
		
		We find such a sequence as follows. Since $(M,g)$ is stably causal, there exists a $C^2$ $g$-temporal function $t : M \to \R$ \cite{bernal2005smoothness} with $grad_g\, t$ future-directed timelike with respect to $g$. Since $\Sigma$ is compact, we can find precompact open neighborhoods $V$ and $U$ of $\Sigma $ in $M$ such that $\overline{V} \subseteq U.$ Consider a smooth bump function $\phi$ that is constant to 1 in $\overline V$ and with support in $U.$ By defining $f = \phi t,$ then $f$ is zero outside of a compact set, and has the same $g$-gradient as $t$ for points of $\Sigma,$ so we define $f_n = f/n$ and $g_n = e^{2f_n}g.$ Since $f_n\to 0$ in the $C^2$ topology      {(see e.g. \cite[p.43-44]{golubitsky_stable_1973}}),  we see that $e^{2f_n}$ converges to 1 in the $C^2$ topology, and then also  $g_n \to g$ in the $C^2$ topology.
		
		To see that $g_n$ indeed satisfies (\ref{eq:FA=Aset}), we make a few elementary computations. By \eqref{eq:MCVconformalscale},
		\[
		\begin{aligned}
			g_n(H_n, H_n) ={}&{} e^{-2f_n}g(H^\Sigma(g), H^\Sigma(g)) - \frac{2me^{-2f_n}}{n}g(H^\Sigma(g),grad_g\, t)\\
			& +  \frac{m^2e^{-2f_n}}{n^2}g((grad_g\, t)^\perp,(grad_g\, t)^\perp).
		\end{aligned}
		\]
		Since $grad_g\, t$ future-directed and $H^\Sigma(g)$ is either zero or past-directed (with respect to $g$), one has $g(H^\Sigma(g),grad_g\, t) \geq 0,$ and $g(H^\Sigma(g), H^\Sigma(g))\leq 0,$ but also the strict inequality $g((grad_g\, t)^\perp,(grad_g\, t)^\perp) < 0$, and thus $	g_n(H_n, H_n) < 0.$ Similarly, by \cref{eq:confscale2},
		\[
		g_n(H_n,X) = e^{-2f_n}g(H^\Sigma(g),X) -  \frac{me^{-2f_n}}{n}g((grad_g\, t)^\perp, X),
		\]and since $g(H^\Sigma(g),X) \geq 0$ and $g((grad_g\, t)^\perp, X) < 0$ because $grad_g\, t$ is also future-directed, we have obtained the desired sequence.	
		
		$\mathcal {A} \subseteq \mathcal {FA}$ is $C^2$-open in $\Gamma^2(L)$, so we immediately have $\mathcal A \subseteq int({\mathcal{F A}}).$ Consider now $ g \in \mathcal {FA} \setminus \mathcal A.$ We shall exhibit a sequence $g_n \notin \mathcal{ F A}$ with $g_n \to g$ in the $C^2$ topology, so that every $C^2$-open neighborhood of $g$ will have a metric not in $\mathcal {FA},$ meaning $g \notin int({\mathcal{FA}}),$ whence we conclude that $\mathcal {FA} \subseteq \overline{\mathcal {A}}$.   
		
		Since $g$ is not a metric in $\mathcal {A},$ there exists a point $p \in M$ for which either $g(H^\Sigma_p, H^\Sigma_p) \geq 0,$ or else $g(H^\Sigma_p, X_p) \leq 0.$ But  $g \in \mathcal {FA},$ so either $H_p^\Sigma(g)$ is lightlike, or otherwise $H^\Sigma_p(g) = 0.$
		
		For the lightlike case, choose a past-directed lightlike vector $v \in T_p \Sigma ^\perp$ not collinear to $H^\Sigma_p$ (which exists because $k=codim \, \Sigma \geq 2$). We easily find some relatively compact neighborhood $U$ of $p$ in $M$ and a smooth function $\phi \in C^\infty(U)$ such that $grad_g \phi|_p = v.$ We then extend $\phi$ globally to a smooth real-valued function $M$ with support in $U$ by a standard bump function argument, and define the sequence $\phi_n = \phi/n$.  This sequence converges to the zero function in the $C^2$ topology, so that $e^{2\phi_n}$ converges to the constant function 1 in the $C^2$ topology, and thus $g_n = e^{2\phi_n}g$ converges to $g$ in the $C^2$ topology on $\Gamma^2(L)$. Denoting by $H^n$ the mean curvature vector of $\Sigma$ associated with the metric  $g_n,$ we can now use (\ref{eq:MCVconformalscale}). Since $H^\Sigma_p(g)$ is (past-directed) $g$-lightlike, and $grad_g \phi|_p = v$ is also past-directed lightlike and $g$-normal to $\Sigma,$ we obtain
		\[
		g_n(H^n_p, H^n_p) = -\frac{2m}{n}e^{2\phi_n(p)} g(H^\Sigma_p(g), v) > 0.
		\]
		
		The case $H_p^\Sigma = 0$ is very similar: we just choose $v \neq 0$ $g$-spacelike in $T_p\Sigma ^\perp$, and the rest of the argument proceeds analogously, but now applying  (\ref{eq:MCVconformalscale}) we get
		\[
		g_n(H^n_p, H^n_p)  = \frac{m^2}{n}e^{-2\phi_n(p)}g(v,v) > 0.
		\]
		
		In any case, we obtain $g_n \notin \mathcal{FA},$ as desired.
	\end{proof}

	\begin{remark}\label{rmk2}
		{\em Some important remarks are appropriate here. 
			\begin{enumerate}
				\item By time-duality, we obviously have analogous results for {\it past-trapped} submanifolds, with an analogous notion of {\it weakly past-trapped} submanifolds.
				\item Consider now metrics in $\Gamma^\infty(L)$, that is, smooth Lorentz metrics on $M$. Then $\mathcal A$ is still $C^\infty$-open, $\mathcal {FA}$ is $C^\infty$-closed, and $\overline{\mathcal A} = \mathcal{FA}$ ($C^\infty$ closure). To see this, note that the proof of  \cref{teo:FAnowhere}, because $\phi$ had compact support actually the sequence $\phi_n$ of smooth functions converges in any $C^r$ topology ($r$ finite) to $0,$ and thus $g_n$ converges to $g$ in any $C^r$ topology. Therefore $g_n \to g$ in the $C^\infty$ topology, and \cref{teo:FAnowhere} can be restated for the $C^\infty$ topology. The exact same kind of proof inspection in \textcite{lerner_space_1973} gives an analogous ``$C^\infty$ statement''  for \cref{teo:riccinowwhere} provided  one considers only smooth metrics. % While we have 
			\end{enumerate}
		}
	\end{remark}
	
	In \textcite{lerner_space_1973} the question was raised that perhaps $\overline{\mathcal{SE}} = \mathcal E$, but the issue was left open in that reference. To the best of our knowledge, this question does not seem to have been addressed in the literature ever since. However, using later results on $C^1$-stability of causal geodesic completeness in globally hyperbolic spacetimes (\textcite[cor. 7.37]{BeemGLG}), we are now able to give a very simple counterexample showing one may indeed have $\mathcal E\setminus \overline{\mathcal{SE}}\neq \emptyset$. 
	
	\begin{example}\label{rmk1.1} 
		{\em In the usual Minkowski spacetime $\R^{m+1}_1$ consider the quotient manifold  $M = \R^{m+1} / \Z ^{m}$ with induced metric $g$,  where the $ \Z ^{m}$ isometric action defined by $(t, x^1, \dotsc, x^{m}) \sim  (t, x^1 + n_1, \dotsc, x^{m} + n_{m}).$  Thus, $(M,g)$ is a flat ($g \in \mathcal E$), globally hyperbolic geodesically complete spacetime. The spacelike Cauchy hypersurface $\Pi = \{t = 0\}$ in  $\R^{m+1}$ is an $m$-torus and thus compact. Since global hyperbolicity, causal geodesic completeness and the spacelike character of $\Pi$ are all $C^1$-stable via the cited result, there is a open $C^2$ neighborhood $\mathcal U$ of $g$ such that all metrics in $\mathcal U$ satisfy these properties. Now, if we had $g \in \overline{\mathcal{SE}},$ then there would exist some $h \in\mathcal{SE} \cap \mathcal U.$ However, in that case the spacetime $(M,h)$ would satisfy all the conditions in the well-known Hawking-Penrose singularity theorem (conf. \textcite[Thm. 12.47]{BeemGLG}), while being causally geodesically complete, a contradiction. Therefore $g \in  \mathcal E\setminus \overline{\mathcal{SE}}.$
		}
	\end{example}	
	
	%While the statement of \cref{teo:FAnowhere} is an analogous version of \cref{teo:riccinowwhere} in the context of trapped and weakly trapped submanifolds, and the idea of the proof is also an adaptation of conformal perturbation techniques used by \textcite{lerner_space_1973}, in \cref{rmk1.1} we have constructed a situation where  $\mathcal E^r \neq \overline{\mathcal{SE}^r}.$ But in the setting of \cref{teo:FAnowhere}, since we are dealing with stably causal metrics, we were actually able to prove $\mathcal {FA} = \overline{\mathcal A}$ with a similar sequence of conformal perturbations argument.

	\subsection{Singularities in codimension two}\label{sect:sing1}
	
	We can now combine these results with the ``energy'' conditions for the Ricci tensor (\cref{teo:riccinowwhere}). As pointed out in \cref{rmk2}(ii), the results in \cref{teo:riccinowwhere} and  \cref{teo:FAnowhere} can be stated for the $C^\infty$ topology. Because the latter assumption is the  most common one in geometry, the following discussion will be carried out for smooth metrics, while easily restated for $C^r$ metrics with the $C^r$ topology with $r\geq 2$.
	
	Recall that the set $\mathcal A$ is $C^\infty$-open in $\Gamma^\infty(L)$, and so is $\mathcal{SE}$. Therefore, $\mathcal M := 
	\mathcal A \cap \mathcal {SE}$ is also $C^\infty$-open in $\Gamma^\infty(L)$, and contained in $\mathscr {ST} (X)\cap \mathscr{S}_\Sigma \cap \mathcal{SC}.$  
	
	Let $\mathcal {FM} := \mathcal{FA} \cap \mathcal E$ be the set of smooth spacetime metrics which are both weakly trapped and satisfy the timelike convergence conditions. This set is $C^\infty$-closed in $\mathscr {ST} (X)\cap \mathscr{S}_\Sigma \cap \mathcal{SC} $; thus $\overline{\mathcal M} \subseteq \mathcal {FM},$ and also 
	\[
	int(\mathcal{FM}) = int(\mathcal{FA}) \cap int(\mathcal{E}) =\mathcal M,
	\] 
	therefore $\mathcal{M}$  is prevalent in $\mathcal{FM}.$

		\begin{remark}\label{rmknice}
			{\em Just as in the discussion showing that $\mathcal{E}\setminus \overline{\mathcal{SE}}$ is not necessarily empty given in example \cref{rmk1.1}, the same spacetime given therein works to show that $\mathcal{FM}\setminus \overline{\mathcal{M}}$ may also not be empty: just consider in addition the spacelike surface originating from the quotient of $S = \{t = x^1 = 0\}$. Therefore $g \in \mathcal{FM},$ but the possibility that $g \in \overline{\mathcal M}$ leads to an analogous contradiction.}
		\end{remark}

%		
%		This counterexample can now be used to display a similar phenomenon as that exhibited in \cref{ex:referee}. %\textcolor{red}{MAS VOCÊ NÃO TINHA FICADO DE CHECAR QUE $\mathcal E$ E $\mathcal{FM}$ ERAM DE BAIRE? NESSE CASO O RESULTADO FICA MUITO MELHOR: $\mathcal{M}$ E $\mathcal{SE}$ SÃO PREVALENTES MAS NÃO-GENÉRICOS EM $\mathcal{FM}$ e $\mathcal{E}$ RESPECTIVAMENTE!}
%		
%		\begin{example}
%			Denote by $N$ the underlying manifold in \cref{rmknice}. We always have $\mathcal M$ prevalent in $\mathcal {FM},$ in particular $\mathcal {FM} \setminus \overline{\mathcal M}$ is nowhere dense in  $\mathscr {ST} (X)\cap \mathscr{S}_\Sigma \cap \mathcal{SC}.$ However, we can show that $\mathcal {FM} \setminus \overline{\mathcal M}$ is not nowhere dense in the subspace topology of $\mathcal{FM}$ for this example.	
%			
%			In fact, since there is some $g \in \mathcal {FM} \setminus \overline{\mathcal M},$ consider an open neighborhood $U$ of $g$ such that $U \cap \overline{\mathcal M}  = \varnothing.$ The set $V = U \cap \mathcal {FM}$ is now an open neighborhood of $g$ in $\mathcal{FM},$ and clearly $V \subseteq  \mathcal {FM}\setminus\overline{\mathcal M};$ thus,  it cannot be nowhere dense in $\mathcal{FM},$ as claimed. 
%		\end{example} 

		Our considerations in these examples can be summarized as follows.
	
	\begin{proposition}\label{prop:conterexprops}
		Both $\mathcal E$ and $\mathcal{FM}$ are Baire spaces in the $C^2$ topology. Moreover
		\begin{enumerate}
			\item[(i)]  $\mathcal{SE}$ is prevalent in $\mathcal E,$ but in general it is not residual in $\mathcal E$ (in the subspace topology); 	
			\item[(ii)]	Analogously, $\mathcal{M}$ is prevalent in $\mathcal {FM},$ but in general it is not residual in $\mathcal{FM}.$
				\end{enumerate}
	\end{proposition}
	\begin{proof}
		The fact that  $\mathcal E$ and $\mathcal{FM}$ are Baire spaces follows from their functional forms as pointwise closed inequalities: these are preserved by pointwise convergence of 2-jets of metrics, and standard properties of weakly closed sets etablish they are indeed Baire in the strong topology (conf. \cite{hirsch_differential_1976}, Thms. 4.2-4.4).
		
		For item (i) (the case (ii) is entirely analogous), we saw above that $\mathcal{SE}$ is prevalent in $\mathcal E$  (\cref{teo:riccinowwhere}) but by \cref{rmk1.1}, we can indeed have $\mathcal E \neq \overline{\mathcal{SE}}$ in general. But in this case, since $\mathcal E$ is a Baire space, $\mathcal{SE}$ cannot be residual, for otherwise it would also be dense; however, $\overline{\mathcal{SE}} = \overline{\mathcal{SE}} \vphantom{E}^\mathcal E \neq \mathcal E.$

	\end{proof}
	
	Now, consider the following key observation. Assume that $(i)$ $g\in \mathcal{M}$ and $(ii)$ $\Sigma$ is a closed submanifold of codimension $k=2$ (following the physics usage we simply say that $\Sigma$ is a \emph{closed surface} in this case). Then, since in particular we will have that $(M,g)$ is chronological and $Ric(g)(v,v)>0$ for any $g$-causal vector, the generic condition for causal vectors will be trivially satisfied (conf. \textcite[Prop. 2.12, p. 39]{BeemGLG}). But then all the hypotheses in the Hawking-Penrose singularity theorem \cite{HawkPen} hold and therefore $(M,g)$ must possess at least one incomplete causal geodesic\footnote{For this particular result, if $(M,g)$ is chronological it is enough that $g\in \mathcal{FA}\cap \mathcal{SE}$ (\textcite{Silva_MOTS_2012}).}.
	
	Now, since all the metrics we consider are restricted to the set $\mathscr {ST} (X)\cap \mathscr{S}_\Sigma \cap \mathcal{SC}$ (so they are in particular chronological), for every $g \in \mathcal M,$ the spacetime $(M,g)$ is causally geodesically incomplete. Since this set is prevalent in $\mathcal {FM}$, we conclude that ``nearly all'' metrics in $\mathcal {FM}$ are arbitrarily $C^\infty$-close to causally incomplete metrics. More precisely, we summarize the discussion of this subsection in the following theorem.
	
	\begin{theorem}\label{mainthm1}
		Let $(M,g)$ be a spacetime of dimension $\geq 3$ and smooth metric containing a spacelike closed surface $\Sigma$. Assume that the following conditions hold. 
		\begin{enumerate}
			%[label=\textnormal{(\roman*)}]
			\item $\Sigma$ is weakly future-trapped;
			\item $Ric(g)(v,v) \geq 0$ for all $g$-timelike $v\in TM$;
			\item $(M,g)$ is stably causal. 
		\end{enumerate} 
		There exists a prevalent set $\chi$ of $\mathcal{FM}$ such that if $g\in \chi,$ then there exists a net $(g_\lambda)_{\lambda \in \Lambda}$ of smooth metrics on $M$ such that $g_\lambda \rightarrow g$ in the $C^\infty$ topology and such that for each $\lambda \in \Lambda$, the spacetime $(M,g_\lambda)$ satisfies the following conditions:
		\begin{enumerate}
			%[label=\textnormal{(\alph*)}]
			\item[(i)] $\Sigma$ is future-trapped;
			\item[(ii)] $Ric(g_\lambda)(v,v) > 0$ for all $g_\lambda$-causal $v \in TM$;
			\item[(iii)] $(M,g_\lambda)$ is stably causal.
		\end{enumerate}
		In particular, $(M,g_\lambda)$ has at least one incomplete causal geodesic.  \hfill $\ensuremath{\Box}$
	\end{theorem}
	%\begin{remark}
	%	As an extra result,	since stable causality is $C^0$ generic in chronology, we can adapt the latter situation to obtain a     ``mixed topology'' version that includes chronology: considering the $C^0$ topology in the $C^2$ open set   $\mathscr{S}_\Sigma \cap \mathscr{ST},$ now the set of $C^2$ metrics 
	%\end{remark}
	
	%	Since stable causality is $C^0$ generic in chronology, we also have a $C^0$ 

	\section{Weakly trapped submanifolds and prevalence of singularities II: higher codimensions}\label{sectmain2}
	
	Our goal now is to establish a version of Theorem \ref{mainthm1} that is valid for all higher codimensions $k\geq 2$ of $\Sigma^m\subset M.$ Since singularity theorems in the presence of higher co-dimensional submanifolds are more delicate, we must introduce some further notation and results as to have a better control of the Riemannian curvature as well. 
	
	For a Lorentz metric $g\in \Gamma^2(L)$ on $M$ denote by $Riem(g)$ the covariant Riemann curvature $(0,4)$-tensor associated with $g.$ Applying a jet bundle argument completely analogous to the one used to prove continuity of the Ricci tensor $Ric$ as a function of $g$ (conf. \textcite[23-24]{lerner_space_1973}),  one readily sees that the function 
	\[
	Riem: g \in \Gamma^r(L) \mapsto Riem(g) \in \Gamma^0(T^{(0,4)}M)
	\]  
	is continuous with respect to the $C^r$-topology on $\Gamma^r(L)$ for each $2\leq r\leq \infty$, where now $T^{(0,4)}M$ denotes the smooth vector bundle of $(0,4)$-tensors over $M$. Consider in $\Gamma^2(L)$ the subset
	\[
	\mathcal P = \{g \in \Gamma^2(L) : Riem(g)(w,v,v,w) >0 \text{ $\forall v \in TM$ $g$-causal and all $w$ non-collinear to $v$}\}.
	\]
	
	\begin{proposition}\label{mainprop2}
		The set $\mathcal P$ is $C^2$-open in $\Gamma^2(L)$. 
	\end{proposition}
	\begin{proof}
		Our arguments adapt some ideas in the proof of \textcite[Prop. 4.3]{lerner_space_1973}, but with a number of modifications of detail. Denote by $Curv(M)$ the vector subbundle of $T^{(0,4)}M$ of all {\it curvature-like tensors} over $M$, i.e., if $F\in Curv(M)$, then for all $p \in M$ and all vectors $w,x,y,z \in T_pM$ we have
		\begin{itemize}
			\item[CL1)]$F(w,z,x,y) = -F(w,z,y,x)$;
			\item[CL2)]$F(w,z,x,y) = -F(z,w,x,y)$;
			\item[CL3)] $
			F(w,z,x,y) + F(w,x,y,z) + F(w,y,z,x) =0$; 
			\item[CL4)]$F(x,y,w,z) = F(w,z,x,y)$. 
		\end{itemize}
		Fix $g \in \mathcal{P}$. Denote by $\widehat {TM}$ the tangent bundle minus all zero vectors, and consider, for $C \in \Gamma^0(Curv(M)),$ the set
		\[
		V_C = \{v \in \widehat{TM} : C(w, v,v,w) >0 \text{ for all $w$ non-collinear to $v$}\}.
		\]  
		We argue that $V_C$ is open in $\widehat{TM}$. Indeed, suppose not. Then there is a sequence $v_n \in \widehat{TM}\setminus V_C$ converging to some vector $v \in V_C$. For each $v_n,$ not being an element of $V_C$ implies that there exists $w_n$ not collinear to $v_n$ and based at the same points respectively such that $C(w_n,v_n,v_n,w_n) \leq 0.$ Now, fix some background Riemannian metric $h$ for $M$. By taking into account the symmetries for $C$ we can assume without loss of generality that all vectors are $h$-unitary, and each $w_n$ is $h$-normal to $v_n.$ By the compactness of the $h$-sphere bundle over compact subsets of $M$ (since $v_n$ is convergent), passing to a subsequence if necessary we can assume that $w_n$ converges to a nonzero $w$ with same base point in $M$ as $v,$ and such that $h(v,w) = 0,$ implying non-collinearity. Thus, on the one hand, $v \in V_C,$ and hence $C(w,v,v,w) > 0,$ and on the other hand the convergence implies $C(w,v,v,w) \leq 0,$ a contradiction.
		
		Denote by $\mathcal C_g \subseteq \widehat{TM}$ the set of all $g$-causal vectors. Now, since $g \in \mathcal P,$ we have $\mathcal C_g \subseteq V_{Riem(g)}$. However, $\mathcal C_g$ is evidently closed in $\widehat{TM}$, therefore, since the latter is a manifold and hence a normal topological space, there exists an open set $U \subseteq \widehat{TM}$ satisfying $\mathcal C_g \subseteq U \subseteq \overline U \subseteq V_{Riem(g)}$. In addition, we can assume without loss of generality that $U$ can be chosen to satisfy $v \in U \implies \alpha v \in U,$ for nonzero $\alpha$ since both sets $\mathcal C_g$ and $V_{Riem(g)}$ have this property. 
		
		Denoting as $\pi : Curv(M) \to M$ the bundle standard projection, we now prove the following statement for this chosen $U$: for each point $p \in M,$ there exists an open set $W \subseteq Curv(M)$ containing $Riem(g)(p)$ with the property that, if for $C \in \Gamma^0(Curv)$ and all $q \in \pi(W)$ we have $C(q) \in W,$ then $\overline U \cap T_q M \subseteq V_C\cap T_q M.$
		Suppose by way of contradiction that this is false. Thus we have a nested sequence of sets $W_n \supseteq W_{n+1},$ all neighborhoods of $Riem(g)(p)$ in $Curv(M)$ with $\bigcap W_n = \{Riem(g)(p)\},$ a sequence $C_n \in \Gamma^0(Curv)$ and points $q_n \in \pi(W_n)$ such that there exists some $v_n \in (\overline U \setminus V_{C_n})\cap T_{q_n}M,$ which means that there is some $w_n\in T_{q_n}M$ not collinear to $v_n$ satisfying $C_n(q_n)(w_n, v_n, v_n, w_n) \leq 0.$ Since $C_n(q_n) \to Riem(g)(p)$ in $Curv(M)$ and $q_n \to p$ in $M$, an argument using a background Riemannian metric similar as the one in the first part of this proof shows that $v_n \to v$ and  $w_n \to w$ up to a subsequence, for some nonzero $v \in \overline U \cap T_p M,$ and $w \in T_p M$ not collinear to $v$. Thus, on the one hand, taking limits we get $Riem(g)(p) (w, v, v, w) \leq 0,$ and on the other hand $v \in \overline{U} \subset V_{Riem(g)},$ so $Riem(g)(p)(w, v, v, w) > 0,$ which is the desired contradiction.
		
		The rest of the proof is quite similar to the final part that of \textcite[Prop. 4.3, p. 28]{lerner_space_1973}. Namely, using the local triviality of the bundle $Curv(M)$ and the neighborhoods $W$ above, we can obtain a $C^0$ neighborhood $\mathcal{W} = \mathcal W(Riem(g)) \subset \Gamma^0(Curv(M))$ of $Riem(g)$ with the property that for any $C \in \mathcal W$ we have $\overline{U}\subset V_C$. Now, by the continuity of the map $Riem$ we have that $\mathcal Z := Riem^{-1}(\mathcal{W})$ is a $C^2$-open set containing $g$ and by construction $\mathcal Z\subset \mathcal{P}$ as desired. 
	\end{proof}
	
	Now, let us define
	
	\[
	\mathcal {FP} = \{g \in \Gamma^2(L) : Riem(g)(w,v,v,w) \geq 0 \text{  $\forall v \in TM$ $g$-causal and all $w$ non-collinear to $v$}\}.
	\]
	If we have a net $(g_\lambda)_{\lambda \in \Lambda}$ in $\mathcal  {FP}$ converging in $C^2$ to $g$, we have that $Riem(g_\lambda) \rightarrow Riem(g)$ in the $C^0$ topology on $\Gamma^0(T^{(0,4)}M)$, whence we conclude that $g\in\mathcal{FP}$; that is, $\mathcal{FP}$ is $C^2$-closed. We now have:
	\begin{theorem}\label{teo:rie}
		In the $C^2$ topology on $\Gamma^2(L)$ we have $int (\mathcal{FP}) = \mathcal P,$ so $\mathcal P$ is prevalent in $\mathcal{FP}.$ %In particular, $\mathcal {FP} \setminus \overline{\mathcal P}$ is nowhere dense in $\mathcal{FP}$.
	\end{theorem}
	\begin{proof}
		Clearly, $\mathcal{P}\subset int(\mathcal{F})$. To prove the other inclusion, we adapt the arguments in the proof of \cref{teo:FAnowhere} above. Specifically, we fix $g \in \mathcal {FP} \setminus \mathcal{P},$ and build via conformal rescalings a sequence $g_n \notin \mathcal{FP}$ of metrics converging in $C^2$ to $g$. We will use the standard formulas for how global conformal rescalings of the metric change the Riemann curvature tensor. (See e.g. \textcite[Eq. 7.44]{leeRiemann}. Although presented for Riemannian metric, the formulas do remain valid in this Lorentzian context.)

		Since $g \notin \mathcal {P},$ we can pick a $g$-causal vector $v \in T_p M$ for some $p \in M,$ and some $w \neq 0$ at the same base point not collinear to $v$ with $Riem(g)(w,v,v,w) = 0.$ Assume first that $v$ is $g$-timelike, so $w$ can be assumed to be spacelike and normal to $v,$ and both can also be assumed to be $g$-unit vectors. Choose a $g$-normal neighborhood system $(x^1, \ldots, x^{dim\, M})$ centered at $p,$ which can be chosen such that $v$ has components $(1, 0,\dotsc, 0),$ and also such that first component of $w$ is zero. Using such coordinates, consider the function $\xi(x^1,\dotsc, x^{dim \, M}) := e^{x^1}.$ We globally extend $\xi$ with usual bump function arguments, (denoting he extension the same way) so that $\xi$ is, in coordinates, $e^{x^1}$ around $p,$ and zero outside some compact subset of $M$ containing $p.$ Define, for each $n\in \mathbb{N}$, $\xi_n = \xi/n,$ $g_n = e^{2\xi_n}g,$ so that again $g_n \to g$ in the $C^2$ topology. With our choice of coordinates, a straightforward computation gives
		\[
		Riem(g_n)_p(w,v,v,w) = -\frac{e^{2/n}}{n} < 0,
		\] 
		so $g_n \notin \mathcal {FP}$ as desired. The case when $v$ is $g$-lightlike is slightly more involved. In this case we pick normal coordinates $(x^1,\ldots, x^{dim\, M})$ centered at $p$ such that $e_i:=\partial /\partial x^i(p)$ ($i=1, \ldots, dim\, M$) form a $g$-orthonormal basis with $e_1$ timelike and $v= e_1+ e_2$. Let $\ell:= e_1-e_2$. $\ell$ is also null, and not collinear with $v$. Now, write
		\[
		w= a\cdot v + b\cdot \ell + \sum_{i=3}^{dim \, M}w^i\cdot e_i.
		\]
		Since $w$ is not collinear with $v$, at least one of the $b,w^i$ are nonzero, and since the part parallel to $v$ gives no contribution to $Riem(g)(w,v,v,w)$ due to the curvature symmetries we can assume without loss of generality that $a=0$. Observe that the vector $\sum_{i=3}^{dim \, M}w^i\cdot e_i$ is normal to both $v$ and $\ell,$   and we have the possibility that either $w$ is lightlike (when each $w^i\equiv 0$) or spacelike. Consider first the spacelike case.  Define $\xi(x^1, x^2,\dotsc, x^n) = (x^1 + x^2)^2$. Proceeding thenceforth just as in the timelike case, we now obtain
		\[
		Riem(g_n)(w,v,v,w) = -\frac{4}{n}g(w,w) < 0.
		\]
		Suppose now $w$ is lightlike. Rescaling, we can assume $w = \ell$. Thus, define $\xi(x^1, x^2,\dotsc, x^n) = (x^1)^2$. Ckmputing as before we obtain 
		\[
		Riem(g_n)(w,v,v,w) =  -\frac{8}{n} < 0.
		\]
		%		We can also slightly perturb $w$ so that $w^i\neq 0$ for some $i\geq 3$ while keeping it not collinear with $v$. But in that case $w$ is spacelike.  By taking $\xi(x^1, x^2,\dotsc, x^n) = (x^1 + x^2)^2$ and from then on proceeding as the timelike case, we obtain
		%		\[
		%		Riem(g_n)(w,v,v,w) = -\frac{4}{n}g(w,w) < 0,
		%		\]  
		Therefore in each case we conclude that $g_n \notin \mathcal {FP}$ as desired, thus completing the proof. 
	\end{proof}
	
	\begin{remark}\label{remark3}
		Again, by arguments analogous to those in \cref{rmk2}(ii), \cref{teo:rie} remains valid for the $C^\infty$ topology, provided we work with smooth metrics. 
	\end{remark}
	
	The following set is actually more relevant: let $\mathcal O\subset \Gamma^2(L)$ be the the of $2$-differentiable metrics for which the associated {\it tidal force operators} (see \textcite[pp. 35 and 38]{BeemGLG} for definitions and notation) along causal directions are positive semidefinite, that is, for $g \in \Gamma^2(L)$ and $v \in TM$ $g$-causal, the linear operators
	\[
	\begin{cases}
		R^{g}_v : v^\perp \to v^\perp, \quad v \text{ for $g$-timelike,} \\
		\overline{R^{g}}_v : \overline{v^\perp} \to \overline{v^\perp}, \quad v \text{ for $g$-lightlike,}
	\end{cases}
	\]
	are both positive semidefinite. It is straightforward to see that \begin{align}
		\mathcal  P&\subseteq \mathcal{SE}; \label{3}\\   
		\mathcal P &\subseteq \mathcal O \subseteq \mathcal {FP}\subset \mathcal{E}\label{4}
	\end{align}
	Therefore, $\overline{\mathcal P} \subseteq \overline{\mathcal O} \subseteq \mathcal{FP},$ and since $int(\mathcal {FP}) = \mathcal P,$ then $int(\mathcal {FP}) \subseteq \mathcal O,$ and the following corollary of \cref{teo:rie} is immediate.
	
	\begin{corollary}\label{maincor2}
		In the $C^2$ topology, $\mathcal O$ is prevalent in $\mathcal {FP}.$  \hfill $\ensuremath{\Box}$
	\end{corollary}
	\begin{remark}\label{easyrmk}
		{\em It is straightforward to check that $g\in \mathcal{FP}$ if and only if 
			\[
			Riem(g)(w,v,v,w)\geq 0, \forall v\in TM \text{ $g$-timelike and all $w\in v^{\perp_g}$}.
			\]
		}
	\end{remark}
	
	\subsection{Singularities in higher codimensions}\label{sect:sing2}
	
	We consider a similar analysis as the one carried out in  \cref{sect:sing1} when the codimension $k$ of the submanifold $\Sigma\subset M$ is higher than 2. However, we cannot use the Hawking-Penrose singularity theorem in this context, but rather use some analogous singularity theorem in the presence of future-trapped submanifolds of codimension higher than 2. Just such a result has been obtained by  \textcite{Galloway_senovilla}.  Also, similar to what we did in \cref{sect:sing1} (cf. \cref{remark3}), the following discussion uses the $C^\infty$ topology on the set of smooth metrics, and is easily adapted to $C^2$ metrics.
	
	We need only to check that the positive-semidefiniteness of the tidal force operators along causal directions encoded in the set $\mathcal O$  does imply the required conditions. Specifically, let $g \in \mathcal{A}\cap \mathcal{O}$ and assume that $\Sigma$ is closed. Following \textcite{Galloway_senovilla}, let $\gamma:[0,b)\rightarrow M$ be some future-directed causal geodesic with $\gamma(0)\in \Sigma$ and $\gamma'(0)$ normal to $\Sigma$. Consider $e_1,\dotsc, e_m$ some coordinate basis of tangent vectors for $T_{\gamma(0)} \Sigma,$ and let $E_1, \dotsc, E_m$ denote their parallel-transport vector fields along $\gamma$. Write $g_{ab} = g(E_a, E_b)$ (which is constant along $\gamma$). Along this geodesic, since we have assumed that $R_{\gamma'}$ (or $\overline{R_{\gamma '}}$ if $\gamma$ is null) is positive semidefinite in the subspaces spanned by the vectors $E_a,$ the trace
	\[
	g^{ab}Riem(\gamma', E_a, E_b, \gamma') \geq 0.
	\]
	But this is precisely the condition 3.1 in  \textcite{Galloway_senovilla}. Observe also that the inclusions \eqref{4} imply that the positive semidefiniteness of the tidal forces also entails the timelike convergence condition. Therefore, if in addition we assume that $(i)$ $(M,g)$ is also chronological and $(ii)$ satisfies the generic condition for causal vectors, then $(M,g)$ is causally geodesically incomplete by \textcite[Thm. 3]{Galloway_senovilla}. Therefore, the arguments in \cref{sect:sing1} by taking \cref{mainprop2}, \cref{teo:rie}, remark \ref{easyrmk} and the inclusions \eqref{3}, \eqref{4} into account can be now adapted as follows.
	
	\begin{theorem}\label{mainthm2}
		Let $(M,g)$ be a spacetime of dimension $\geq 3$ and smooth metric containing a spacelike closed submanifold $\Sigma$ of codimension $k\geq2$. Assume that the following conditions hold. 
		\begin{enumerate}
			%[label=\textnormal{(\roman*)}]
			\item $\Sigma$ is weakly future-trapped;
			\item $Riem(g)(w,v,v,w) \geq 0$ for all $g$-timelike $v\in TM$ and all $w$ at the same base point $g$-orthogonal to $v$;
			\item $(M,g)$ is stably causal. 
		\end{enumerate} 
		There exists a prevalent set $\chi$ in $\mathcal{FP}$ such that if  $g\in \chi$, then there exists a net $(g_\lambda)_{\lambda \in \Lambda}$ of smooth metrics on $M$ such that $g_\lambda \rightarrow g$ in the $C^\infty$ topology and such that for each $\lambda \in \Lambda$, the spacetime $(M,g_\lambda)$ satisfies the following conditions
		\begin{enumerate}
			%[label=\textnormal{(\alph*)}]
			\item[(i)] $\Sigma$ is future-trapped;
			\item[(ii)] $Ric(g_\lambda)(v,v) > 0$ for all $g_\lambda$-causal $v \in TM$, and thus the causal genericity condition holds in $(M,g_\lambda)$;
			\item[(iii)] the tidal force operators along $g_\lambda$-causal directions are all positive semidefinite;
			\item[(iv)] $(M,g_\lambda)$ is stably causal.
		\end{enumerate}
		In particular, each $(M,g_\lambda)$ has at least one incomplete causal geodesic.	  \hfill $\ensuremath{\Box}$
	\end{theorem}

	\section{Genericity of singularities in Cauchy developments of initial data sets}\label{sectmain3}
	
	We now turn to an investigation of genericity of singular spacetimes (under suitable conditions) representing maximal Cauchy developments of \textit{initial data sets} containing a MOTS. We shall do so by considering a separable Hilbert manifold structure on the set of triples $(h,\mathcal K, \psi),$ where $(h, \mathcal K)$ is an initial data on a fixed connected $n$-manifold $\mathcal{S}$ and $\psi : \Sigma \to \mathcal S$ a MOTS embedding, where again $\Sigma$ is a fixed compact and connected manifold (without boundary). The genericity statement we seek will follow from an abstract functional-analytic approach together with the well-known Sard-Smale theorem.

	\subsection{An abstract Banach manifold genericity result} In order to obtain our main genericity results, we shall adopt a suitable Banach manifold structure both on the set of initial data on a fixed connected $n$-manifold $\mathcal{S}$, and on the set of embeddings of a fixed connected compact $(n-1)$-manifold $\Sigma$ into $\mathcal{S}$. However, we shall first establish in this section a key but purely abstract result on Banach manifolds, via a relatively straightforward adaptation of the main analysis in \textcite{piccione1}. Our first lemma is a direct adaptation of lemma 2.2 in that reference. For later reference, we also list a few other results from \cite{piccione1} that will also be needed.
	%, clearly indicating the source when we do so, as needed. 
	
	Let $W$ be a closed vector subspace in a Hilbert space $H.$ We denote by $P_W : H \to W$ the associated orthogonal projection onto $W.$ 
	
	Recall that a bounded linear operator $T: E \to F$ between Banach spaces is \emph{Fredholm} if  $\ker T$ and $\operatorname{coker}{T}$ have finite dimension\footnote{Some authors explicitly require the range of $T$ to be closed; however, this condition follows easily from the finite dimensionality of the cokernel.} , where  $\operatorname{coker}{T} = E/ \im T$ (the dimension of the cokernel is also called the \emph{codimension} of $\im T$). The \emph{index} of a Fredholm operator $T$ is defined as $\operatorname{ind}(T) = \dim \ker T - \dim \operatorname{coker}{T}.$ For later reference, recall also that if $M,N$ are (perhaps infinite dimensional) Banach manifolds and $a$ is an integer number, then a $C^1$ map $f: M\rightarrow N$ is said to a \textit{Fredholm map of index $a$} if its derivative $df_p:T_pM\rightarrow T_{f(p)}N$ is a Fredholm linear operator of fixed index $a$ for every $p \in M$.

	%\begin{proposition}\label{funcanlsurj}
	%	Let $E, F$ be Banach spaces, and $H$ a Hilbert space with inner product $\inner{\,\,}{\,},$ and consider  $T : E \to H,$  $S : F \to H$ bounded linear operators, with $\im S$ closed. Then the sum operator $T \oplus S : E \times F \to H$ (given by $T \oplus S(x,y) = T(x) + S(y)$) is surjective if and only if $P_{\im S^\perp}(\im T) = \im S^\perp.$
	%\end{proposition}
	%\begin{proof}
	%	Assume $V$ is surjective, and consider some $z \in \im S^\perp,$ so there is a pair $(x,y) \in E \times F$ with $T(x) + S(y) = z,$ therefore $z = P_{\im S^\perp}(z) = P_{\im S^\perp}(T(x)).$
	%	
	%	For the converse, with  $\im S$ closed, $H = \im S \oplus \im S^\perp,$ so given $z \in H,$ $z = S(y) + u,$ $u$ normal to $\im S.$ Since  $P_{\im S^\perp}(\im T) = \im S^\perp,$ $u = P_{\im S^\perp}(T(x)).$ But also, $T(x) = S(\hat y) + P_{\im S^\perp}(T(x)),$ and we obtain
	%	\[
	%	z = T(x)+ S(y) - S(\hat y) = T \oplus S (x, y - \hat y).
	%	\]  
	%\end{proof}
	
	\begin{lemma}\label{funcanlsurj}
		Let $E, F$ be Banach spaces, and let  $H$ be a Hilbert space with inner product $\inner{\,\,}{\,}.$ Consider  $T : E \to H,$  $S : F \to H$ bounded linear operators, with $\im S$ closed and such that $P_{\im S^\perp}(\im T)$ is closed in $\im S^\perp$ (this happens for example if $S$ is a Fredholm operator). Then the sum operator $T \oplus S : E \times F \to H$ (given by $T \oplus S(x,y) = T(x) + S(y)$) is surjective if and only if $\im T ^\perp \cap \im S^\perp = \{0\}.$ 
	\end{lemma}
	\begin{proof}
		Assuming the sum operator to be surjective, given $z \in  \im T ^\perp \cap \im S^\perp,$ then there is a pair $(x,y)$ with  $z = T(x) + S(y),$ but also by perpendicularity $$\|z\|^2 = \inner{z}{z} = \inner{z}{ T(x) + S(y)} = 0. $$ 
		
		Assuming now $\im T ^\perp \cap \im S^\perp = \{0\},$ take some $z \in H\setminus\{0\}.$ With $\im S$ closed, $H = \im S \oplus \im S^\perp,$ so $z = S(x) + u,$ $u \in   \im S^\perp.$ We can also split $ \im S^\perp$ (since it is a Hilbert space under the restriction of the inner product and  $P_{\im S^\perp}(\im T)$ is closed in $\im S^\perp$) as follows: 
		\[
		\im S^\perp = [P_{\im S^\perp}(\im T)] \oplus [(P_{\im S^\perp}(\im T))^\perp\cap 	\im S^\perp].
		\]
		Now, $u = P_{\im S^\perp}(T(x))+ v,$ where $v \in (P_{\im S^\perp}(\im T))^\perp\cap 	\im S^\perp.$ Also, $T(x) = S(\hat y) +  P_{\im S^\perp}(T(x)),$ and we have $z = T(x) + S(y) - S(\hat y) + v.$ If $v = 0$ we are done. Now, $v \in \im S^ \perp,$  and for any $T(e),$ $e \in E,$ again $T(e) = S(f) +  P_{\im S^\perp}(T(e)),$   and we obtain,  $$\inner{v}{T(e)} = \inner{v}{S(f)} + \inner{v}{ P_{\im S^\perp}(T(e))} = 0,$$
		meaning $v \in \im T^\perp,$ implying $v=0.$  
	\end{proof}
	
	\begin{remark}\label{rmk:adjoints}
		{\em For the Banach adjoint $T^* : H^* \to E^*,$ it is a standard textbook result (\textcite{bachman2000functional}, p. 285) that we have $\im T^\perp = \ker T^*,$ so the trivial intersection of images in \cref{funcanlsurj} can be substituted with trivial intersection of the kernels of adjoint maps.}
	\end{remark}
	
	Recall also that in a Banach space $E,$ a closed subspace $V \subseteq E$ is said to be \emph{complemented} if there is another another closed subspace $W \subseteq E,$ called a \emph{complement of $V$,} satisfying $E = V \oplus W.$
	
	\begin{proposition}[\cite{piccione1}, lemma 2.4]\label{complementkernel}
		Let $E, F, G$ be Banach spaces,  and $T: E \to G,$ $S : F \to G$ bounded linear operators, with $\ker S$ complemented in $F$ and $\im S$ finite codimensional in $G$ (this happens for example if $S$ is a Fredholm operator). Then the kernel of $T \oplus S$ is complemented in $E \times F.$      
	\end{proposition}
	\begin{proposition}[\cite{piccione1}, lemma 2.3]\label{absfunc3}
		Let $U, V$ be vector spaces, $L : U \to V$ a linear map, and $S\subseteq V$ a finite codimensional subspace. Then $L^{-1}(S)$ has finite codimension in $U,$ with
		\begin{equation}\label{codimformula}
			\codim_U {(L^{-1}(S))} = \codim_V {(S)} - \codim_V{(S + \im L)}.
		\end{equation}
	\end{proposition}
	%\begin{proof}
	%	Considering the quotient projection $\pi : V \to V / S,$ the composition $\pi \circ L : U \to V / S$ has kernel $L^{-1}(S),$ so the induced quotient map $[\pi \circ L] : U / L^{-1}(S) \to V / S$ is injective, with proves $L^{-1}(S)$ has finite codimension in $U.$ 
	%	
	%	For (\ref{codimformula}), using standard dimensional analysis from linear algebra applied to the linear map $[\pi \circ L] : U / L^{-1}(S) \to V / S$ we have
	%	\[
	%	\begin{aligned}
		%		\codim_V{S} = \dim{(V / S)} &= \dim{(\im [\pi \circ L] )} + \codim_{V/S} {(\im [\pi \circ L])}\\
		%		& = \dim{(U/L^{-1}(S))} + \codim_V{(\im {(L)} + S)}.
		%	\end{aligned}
	%	\]
	%\end{proof}			
	
	For the abstract Banach manifold result we seek, we fix a separable Banach manifold $X$, a Hilbert manifold $Y$, and a Hilbert space $V.$ We apply the general functional-analytic results above to prove the following proposition, an adaptation of \textcite{piccione1}, prop. 3.1 with more suitable hypotheses for our concrete case of interest.
	
	\begin{proposition}\label{prop:absctractbanman}
		Consider $A \subseteq X \times Y$ an open set, and let $f : A \to V$ be a $C^k$ function, $k \geq 1.$ For every $(x_0, y_0) \in A$ such that $f(x_0,y_0) = 0,$ assume that the partial derivative $\dfrac{\partial f}{\partial y}(x_0, y_0) : T_{y_0} Y \to V$ is a Fredholm operator. Then $0$ is a regular value of $f$ if and only if
		\begin{equation}\label{eq:perpcondition}
			\left( \im \frac{\partial f}{\partial x}(x_0, y_0) \right)^\perp \cap 	\left( \im \frac{\partial f}{\partial y}(x_0, y_0) \right)^\perp = \{0\}.
		\end{equation}
		\begin{proof}
			0 is a regular value of $f$ if and only if for all $(x_0, y_0) \in f^{-1}(0),$ $df_{(x_0, y_0)} : T_{x_0}X \times T_{y_0} Y \to V$ is surjective with complemented kernel. Now  $df_{(x_0,y_0)}$ can be written as the sum operator
			\[
			df_{(x_0,y_0)} = \frac{\partial f}{\partial x}(x_0, y_0) \oplus \frac{\partial f}{\partial y}(x_0, y_0).
			\]
			Applying \cref{funcanlsurj} since $\dfrac{\partial f}{\partial y}(x_0, y_0)$ is Fredholm, it follows that if $df_{(x_0,y_0)}$ is surjective  then (\ref{eq:perpcondition}) holds. Conversely if (\ref{eq:perpcondition}) holds,  by  \cref{funcanlsurj} we obtain surjectivity, and also by \cref{complementkernel} the kernel of $df_{(x_0,y_0)}$ is complemented.
		\end{proof}
	\end{proposition}
	
	\begin{remark}\label{rmk:tcondition}
		For the sake of brevity, we shall name the intersection condition (\ref{eq:perpcondition}) as the \emph{T condition}. As stated in \cref{rmk:adjoints}, we can substitute this condition for 
		\[
		\ker {\left( \frac{\partial f^*}{\partial x}(x_0, y_0) \right)} \cap 	\ker {\left( \frac{\partial f^*}{\partial y}(x_0, y_0) \right)}= \{0\}
		\]
		at points $(x_0, y_0) \in f^{-1}(0).$ This might eventually be helpful in concrete situations as it directly states the condition in the form of system of equations for the adjoints. 
	\end{remark}
	
	Under the hypothesis of \cref{prop:absctractbanman}, assuming (\ref{eq:perpcondition}) to be valid, then the level set $M =f^{-1}\{0\} \subseteq A$ is an embedded submanifold or $X\times Y,$ with tangent space at some $(x_0,y_0)\in M$ given by
	\[
	T_{(x_0,y_0)} M = \left\{\, (v,w) \in T_{x_0} X \times T_{y_0} Y : \frac{\partial f}{\partial x}(x_0, y_0)(v) + \frac{\partial f}{\partial y}(x_0, y_0)(w) = 0  \, \right\}.
	\] 
	The main abstract genericity result we will need is the following (adapted from \textcite[corollary 3.4]{piccione1}).
	
	\begin{theorem}\label{teo:banachgen}
		Under the hypothesis of \cref{prop:absctractbanman}, assume also that $\dfrac{\partial f}{\partial y}$ is a Fredholm operator of index zero for points in $M.$ Denoting by $\Pi : X \times Y \to X$ the canonical projection, then,
		\begin{enumerate}
			\item[(i)]  $\Pi|_{M}:M\rightarrow X$ is Fredholm map of index zero.
			\item[(ii)] The critical points of $\Pi|_M$ are the points $(x_0,y_0) \in M$ such that $y_0$ is a critical point of the functional
			\[
			y \in A_{x_0} \mapsto f(x_0, y) \in V,
			\]
			where $A_x = \{y \in Y : (x,y) \in A\}.$ 
			\item[(iii)]  Assuming $X, Y$ and $V$ to be separable, the set of points $x \in X$ such that the functional $y \in A_x \mapsto f(x,y) \in V$ has no critical points is generic in $\Pi(A).$
		\end{enumerate}
	\end{theorem}
	\begin{proof}
		For (i), given $(x_0, y_0) \in M,$ we have $\ker d \Pi|_{T_(x_0, y_0) M} = T_{(x_0, y_0)} M \cap (\{0\} \times T_{y_0} Y),$ and this vector space is isomorphic to  $\ker\left(\frac{\partial f}{\partial y}(x_0, y_0)\right)$ so it is finite dimensional. Also, one readily sees that
		\[
		\im d \Pi|_{T_{(x_0, y_0)} M}  = \left(\frac{\partial f}{\partial x}(x_0, y_0)\right)^{-1}	\left( \im \frac{\partial f}{\partial y}(x_0, y_0)\right),
		\] 
		so from \cref{absfunc3}, $\im d \Pi|_{T_{(x_0, y_0)} M}$ is finite codimensional, and since $df_{(x_0,y_0)}$ is surjective, by \cref{codimformula} 
		\[
		\codim (\im d \Pi|_{T_{(x_0, y_0)} M}) = \codim  \left(\frac{\partial f}{\partial y}(x_0, y_0)\right).
		\]
		Since we are assuming $\frac{\partial f}{\partial y}(x_0, y_0)$ to be Fredholm of index zero, it follows that $d\Pi|_{T_(x_0, y_0) M}$ is also Fredholm of index zero.
		
		For (ii), it is easy to see that $(x_0, y_0) \in M$ is a regular point of $\Pi|_M$ if and only if 
		\[
		\im  \left(\frac{\partial f}{\partial x}(x_0, y_0)\right) \subseteq \im  \left(\frac{\partial f}{\partial y}(x_0, y_0)\right).
		\]
		Taking orthogonal complements, this is equivalent to
		\[
		\im  \left(\frac{\partial f}{\partial y}(x_0, y_0)\right)^\perp \subseteq \im  \left(\frac{\partial f}{\partial x}(x_0, y_0)\right)^\perp,
		\]
		and by condition (\ref{eq:perpcondition}), equivalent to $\im  \left(\frac{\partial f}{\partial y}(x_0, y_0)\right)$ be a zero codimensional subspace, in turn equivalent to the triviality of $\ker  \left(\frac{\partial f}{\partial y}(x_0, y_0)\right).$ 
		
		To summarize,  $(x_0, y_0) \in M$ is a regular point of $\Pi|_M$ if and only if $\ker  \left(\frac{\partial f}{\partial y}(x_0, y_0)\right) = \{0\},$ that is, if and only if  $y_0$ is a regular point of the functional $y \in A_{x_0} \mapsto f(x_0, y) \in V.$ Item (iii) now follows, because the set of points $x \in X$ such that $y \in A_x \mapsto f(x,y) \in V$ has no critical points coincides with the set of regular values for $\Pi|_{M},$ and genericity follows by the Sard-Smale theorem (\textcite{smale1965infinite}).
	\end{proof}

	\begin{remark}			
		We make here an heuristic comment on the plausibility of the hypotheses in \cref{teo:banachgen}.

		The condition $T$ together with the Fredholm condition is equivalent to $0$ being a regular value of a $C^k$  function $f : A \subseteq X \times Y  \to V.$ 
		Now, suppose $X,Y,V$ were all finite-dimensional. Then, the simplest version of elementary transversality theorem (cf. \textcite{mukherjee_differential_2015}, thm. 8.7.7) applied to the  submanifold $W=\{0\}\subset V$ would mean the the set of $C^k$ functions $f:A\rightarrow V$ for which $0$ is a regular value is residual with respect to a Whitney topology. This suggests that the condition $T$ might not be restrictive.
		
		While we expect to have an analogous situation on a suitable class of infinite dimensional manifolds, tranversality theory is much harder in this case, and so we were unable to find a generalization of the Thom transversality theorem applicable in our case. Therefore the $T$ condition will remain a technical condition for us.				
	\end{remark}
	
	%				A question that will be  concern later regarding this last theorem is that, for elements in $x \in \Pi(A),$ is $\Pi^{-1}(x) \cap M \neq \varnothing$?  It is unclear for points in $\Pi(A)$ as originally  stated in  \cref{teo:banachgen},  but we can reduce $A$ around a fiducial $(x_0, y_0) \in M$ and answer the question positively. First take a possibly smaller $A$ of the form $A=B \times C \ni (x_0,y_0).$  	

	%		{	\color{red} Substituir essa parte em amarelo por: "Now, suppose $X,Y,V$ were all finite-dimensional. Then, the simplest version of Thom's transversality theorem (theorem A.3.14) applied to the ``0-jet'' $j^0f = f$ and the submanifold $W=\{0\}\subset V$ would mean the the set of $C^k$ functions $f:A\rightarrow V$ for which $0$ is a regular value is residual with respect to a Whitney topology. This suggests that the condition $T$ might not be restrictive."}

	\subsection{Manifold structures on initial data sets}
	
	The function we will concretely analyze in order to apply \cref{teo:banachgen} is the expansion scalar $\theta_+$ defined on a suitable open set (when manifold structures are defined) around a triple $(h_0, \mathcal K_0, \psi_0),$ with $(h_0, \mathcal K_0)$ an initial data set  for which $\psi_0 : \Sigma \to \mathcal S$ is a MOTS embedding. With a domain well defined, the expansion scalar is a function of the  form
	\[
	(h,\mathcal K, \psi) \mapsto \theta_+(h,\mathcal K, \psi) = \trace_h \mathcal K \circ \psi + H^\psi_{h, \bm \upnu}.
	\]
	
	First,we describe possible infinite dimensional manifold structures on the set of vacuum initial data, then discuss a Hilbert manifold structure for the set of embeddings $\Sigma \to \mathcal S.$  
	
	These infinite dimensional structures are very technical in nature, and describing them precisely is out of the scope of this work. Fortunately, we shall need only very broad features of these here, so we briefly summarize the ideas behind them. When these structures are established, we can define the function $\theta_+$ more rigorously.

	\subsubsection{Manifold of initial data sets}\label{idbanachsctructure}
	
	%			The function we will concretely analyze is the expansion scalar. Denote by $ID(\mathcal S)$ the set of initial data $(h,\mathcal K)$ over $\mathcal S$ satisfying the vacuum constraint equations, and  $Emb_{\mathcal O}(\Sigma, \mathcal S)$ the set of orientable embeddings $\Sigma^{n-1} \to \mathcal S^{n}.$  The null expansion scalar can be viewed as then the mapping
	%			\[
	%				\begin{aligned}
		%				\theta_+ : ID(\mathcal S) \times Emb_{\mathcal O}(\Sigma, \mathcal S) &\to  C^\infty(\Sigma)\\
		%				((h, \mathcal K), \psi) &\mapsto \tr_{h}\mathcal K \circ \psi + H^\psi_{h,\bm{\upnu}}.
		%			\end{aligned}
	%			\]
	%			
	%			
	%			
	
	Denote by $ ID(\mathcal S)$ the set of vacuum initial data $(h,\mathcal K).$ We set to establish a reasonable structure for the initial data set $ID(\mathcal S).$

	Since we want conditions on $\mathcal S$ as to apply \cref{teo:motssingularity}, $\mathcal S$ cannot be compact. For a  concrete manifold structure under such conditions we assume $\mathcal S$ to be a non-compact, oriented \emph{asymptotically flat manifold} without boundary. Asymptotically flatness here means in particular that $\mathcal S$ has a compact set $K$ such that  $\mathcal S \setminus K$ is the union of a finite number of regions $\mathcal S_1,\dotsc, \mathcal S_\ell$ called the \emph{ends} of $\mathcal S,$ each diffeomorphic to $\R^n \setminus B,$ where $B$ denotes the closed ball of radius $1$ centered at zero\footnote{The precise definition of asymptotic flatness also involves detailed falloff conditions for the metric and its derivatives up to second order that need not concern us here.}.  
	
	For the sake of simplicity we assume $\mathcal S$ has only one end. Under these conditions we can find a few descriptions of a Banach manifold structure on the set of initial data on such asymptotically flat manifolds. For our practical purposes, we choose the smooth Hilbert manifold structure laid out by	\textcite{bartnik2005phase}. This structure\footnote{Originally we wanted to model using a \textit{Banach} manifold structure following  \textcite{chrusciel2003mapping}, since their work is not limited to the asymptotically flat case. However their Banach structure is of the form $H^l \cap C^{r,\alpha}$ (both Sobolev and H\"{o}lder) using a Sobolev-H\"{o}lder sum norm, and it is unclear to us if the Banach manifold locally modeled on this space is separable, as this condition is crucial to apply the Sard-Smale theorem.} is locally modeled on Sobolev spaces $H^2 \times H^{1},$ and separable.

	\subsubsection{Manifold of embeddings}
	
	Infinite dimensional structures for functions spaces $\Sigma \to \mathcal S$ with $\Sigma$ compact is a well-established idea in the literature. Here we follow \textcite{alias2011manifold}, section 3. 
	
	The idea of the process is to choose a suitable class of maps of regularity $\mathscr{R}$ lower than $C^ \infty$ regularity, which densely contains $C^\infty$ and can be continuously embedded into those of $C^{2,\alpha}$ regularity   (\textcite{alias2011manifold}, section 2), then model the set $\mathscr{R}(\Sigma, \mathcal S)$ of maps $\Sigma \rightarrow \mathcal S$ of regularity $\mathscr{R}$ as a Banach or Hilbert manifold locally modelled on $\mathscr{R}(\Sigma, \mathbb{R}^N)$, and finally argue that the set $Emb(\Sigma, \mathcal S)$ of embeddings is open in $\mathscr{R}(\Sigma, \mathcal S).$ 
	
	Here we choose Sobolev $H^{k+2}$ regularity, with $k, r \in \N$ and $\alpha \in (0,1)$ (we take $r\geq 2$ since at least two derivatives are computed in curvature tensors) satisfying $(k+2 - r - \alpha)> 1/2$ so that the Sobolev space $H^{k+2}$ is continuously embedded into the H\"{o}lder space $C^{r,\alpha}$ (\textcite{aubin_nonlinear_1998}, thms. 2.10 and 2.20). Thus, we do not lose much regularity by locally modeling the embeddings over Sobolev spaces for large $k$.  		
	More precisely, \textcite{alias2011manifold} show that $Emb(\Sigma,\mathcal{S})$ in this regularity can be viewed as a Hilbert manifold locally modeled\footnote{Integrals on $\Sigma$ are computed with respect to some background Riemannian metric $h$, but the resulting Sobolev space does not depend on this choice. See \textcite{aubin_nonlinear_1998} for details.} on the real Hilbert space $H^{k+2}(\Sigma).$

 %\begin{remark}

 %\end{remark}
	
	%			For a Hilbert manifold structure on the set of embeddings, we choose a standard structure as done in , section 3.

	%			 A key component for their analysis that deserves a brief explanation is that we work under the \emph{no Killing initial data} (the \emph{no KIDs condition}). 

	\subsection{The null expansion scalar $\theta_+$}

	To make technical sense of the expansion scalar as a smooth function on the manifold of initial data sets together with two-sided, framed embeddings therein, one requires an outward normal direction fixed \textit{a priori}. As a consequence we can only work on a neighborhood of a fiducial $(h_0, \mathcal K_0, \psi_0).$ 
	
	To be more precise, consider  $(h_0, \mathcal K_0) \in ID(\mathcal S)$ an asymptotically flat initial data set, and $\psi_0 \in Emb(\Sigma, \mathcal S)$ a smooth embedding that is a MOTS with respect to initial data set  $(h_0, \mathcal K_0),$ with 
	$\bm\upnu_0\in \vecfp{\psi_0}$ the chosen outward pointing unit normal vector field of the embedding $\psi_0$ under $h_0.$
	Then  
	$$\theta_+(h_0,\mathcal K_0, \psi_0) = \trace_{h_0} \mathcal K_0 \circ \psi_0 + H^{\psi_0}_{h_0, \bm \upnu_0} = 0.$$
	The process is now to ``propagate'' the normal vector $\bm \upnu_0$ over a small neighborhood around $\psi_0(\Sigma).$ Since  $\psi_0(\Sigma)$ is a smooth compact submanifold of $\mathcal S,$ there is a smooth vector field $\mathcal V \in \vecf{\mathcal S}$ such that $\mathcal V\circ \psi_0 = \bm\upnu_0.$ Restricting to a small enough neighborhood around $\psi_0(\Sigma),$ we can assume $\mathcal V_p \notin (d \psi_0)_p(T_p\Sigma)$ at all points in this neighborhood. We can then choose a open neighborhood around our embedding $\psi_0$ in $Emb(\Sigma,\mathcal S)$ such that all embeddings $\psi$ in this neighborhood have $\mathcal V\circ\psi$ nowhere tangent. (These embeddings are called \emph{framed embeddings}, and  we see via this argument that this is a stable property). With initial data $(h, \mathcal K)$ close enough to $(h_0, \mathcal K_0),$ we consider the $h$-normal part of $\mathcal V\circ \psi:$ $$\bm\upnu(h,\psi)  = \frac{(\mathcal V\circ \psi)^{\perp,h}}{\|(\mathcal V\circ \psi)^{\perp,h}\|_h}.$$
	Denote by $\mathcal A_0$ the open set in $ID(\mathcal S)\times Emb(\Sigma,\mathcal S)$ around $(h_0, \mathcal K_0, \psi_0)$ as describe above. Then we can define 
	\begin{equation}\label{eq:thetafuncdef}
		\begin{aligned}
			\theta_+ : (h, \mathcal K, \psi) \in \mathcal A_0  \mapsto \trace_h \mathcal K\circ \psi + H^\psi_{h,\bm\upnu(h,\psi)},
		\end{aligned}
	\end{equation}
	
	%			 We now consider a neighborhood $\mathcal A$ of $(h_0,\mathcal K_0, \psi_0)$ in $ID(\mathcal S ) \times Emb(\Sigma, \mathcal{S})$ 

	\subsection{Linearization of $\theta_+$}
	
	With $\theta_+$ suitably defined, the formal partial linearization of $\theta_+$ with respect to the embedding is the MOTS stability operator (\cref{def:OpStability}): if $\psi : \Sigma \to \mathcal S$ is a MOTS embedding for the initial data $(h,\mathcal K),$ we have
	\[
	\left.	\frac{\partial \theta_+}{\partial \psi}\right|_{(h, \mathcal K,\psi)}(f) = L(f) =   - \Delta_h f + 2 \inner{X} {\grad_h f}_h + (Q  
	+ \operatorname{div}_h X   - \| X \|_h^2 ) f,
	\] where
	\begin{equation*}\label{eq:Qdefinition_INITDATAsddsadas}
		Q = \frac{1}{2} \mathrm{Scal}_\Sigma -[J(\bm{\upnu})+\rho] -  \frac{1}{2}| \mathcal{K}_{\bm{\upnu}} + \mathcal{K}|^2,
	\end{equation*}
	and $X$ the vector field $h$-dual to the one form $\mathcal K (\bm\upnu, \cdot {})$
	%			All the geometric objects her are defined on $\Sigma$, $\bm{\upnu}$ is the outward pointing unit normal vector field on $\Sigma,$ $\mathcal K_{\bm{\upnu}}$ is the scalar second fundamental form of $\Sigma$ with respect to the induced metric from $(\mathcal S, h)$ on the direction $\bm{\upnu},$  $X$ is the vector field dual to the one-form $\mathcal{K}(\bm{\upnu}, \cdot)$ on $\Sigma$. Finally, $\rho$ and $J$ are defined as in \cref{def:InitialDataSet}.
	
	An issue that needs addressing in other to apply the analytical machinery developed is the Fredholm property for the MOTS stability operator. However, since $L$ is a second order linear eliptic operator with a specific form, we have the following general result.
	
	\begin{proposition}[\textcite{danlee}, cor. A.9]Let $(\Sigma,h)$ be a compact Riemannian manifold without boundary, and let $L$ be a second order linear elliptic operator on $(\Sigma,h)$ of the form
		\[
		L(f) = -\Delta_h f + \inner{V}{grad_h f}_h + qf,
		\]
		with $V \in \vecf{M}, q \in C^{\infty}(\Sigma).$ Then
		\[
		\begin{aligned}
			&L : W^{k+2,p}(\Sigma) \to  W^{k,p}(\Sigma),\\
			&L : C^{k+2,\alpha}(\Sigma) \to C^{k,\alpha}(\Sigma)
		\end{aligned}
		\]
		are Fredholm operators of index zero for each $k \in \N.$  \hfill $\ensuremath{\Box}$
	\end{proposition}

	\subsection{Geometric interpretation}

	With the abstract analytic machinery in place, a suitable infinite dimensional manifold structure and $\theta_+$ all well established, we now give our main results. 
	To simplify our language, let us introduce some terminology: say that a triple $(h, \mathcal K, \psi)$ is a \emph{MOTS triple} if $\theta_+(h, \mathcal K, \psi)=0$ (i.e. the embedding $\psi$ is a MOTS in the initial data $(h, \mathcal K)$),  and call a MOTS triple $(h, \mathcal K, \psi)$  \emph{non-degenerate} if the associated MOTS stability operator has a non-zero principal eigenvalue (in such case the stability operator is said to be \emph{non-degenerate}). Also, recall that we are working with \emph{vacuum} initial data since this is a concrete case where a separable Hilbert manifold structure exists\footnote{Due to the abstract nature of \cref{teo:banachgen} one could adapt our main results to other sets of non-vacuum initial data if they admit some separable Banach manifold structure, but such are hard to obtain so we restrict ourselves to vacuum initial data.
 following \cite{bartnik2005phase}, so all initial data in what follows is assumed to be vacuum initial data.}
	
	\begin{theorem}\label{teo:mainMOTS1}
		Let $(h_0, \mathcal K_0)$ be  initial data set, %satisfying the no KIDs condition
		$\psi_0 : \Sigma \to \mathcal S$ an embedding that is a MOTS for the initial data $(h_0, \mathcal K_0)$ on $\mathcal{S},$ and consider the neighborhood $\mathcal A_0$ where $\theta_+$ is defined. Assume $\theta_+$ satisfies the T condition (\cref{rmk:tcondition}).  Then there exists an open neighborhood $\mathcal O \subseteq ID({\mathcal S})$ around $(h_0,\mathcal K_0)$ where the set
		\begin{equation}\label{eq:motsgenericset}
			\mathcal G= \{(h, \mathcal K) \in \mathcal O : \textnormal{every MOTS in $(h, \mathcal K)$ is non-degenerate} \}
		\end{equation}
	\end{theorem}
	is  a residual set in $\mathcal O.$  
	\begin{proof}			
		Since the MOTS stability operator is a Fredholm operator of index 0, we are in conditions to apply \cref{teo:banachgen}: here, $M = \theta_+^{-1}(0)$ is the submanifold in
		$\mathcal A_0$  of MOTS triples. By item (iii), it is a generic property in $\Pi(\mathcal A_0) = \mathcal O$ the following: initial data $(h,\mathcal K)$  such that the functional
		\[
		\psi \in \mathcal (A_0)_{(h,\mathcal K)} \mapsto \theta_+(h,\mathcal K, \psi)
		\] 
		has no critical points. Having no critical points, then any MOTS embedding  $\psi$ for such $(h,\mathcal K)$ is such that the associated stabability operator is non-degenerate. 
	\end{proof}

	\begin{remark}\label{rmk:veryimportant}
		What \cref{teo:mainMOTS1} says about the residual set $\mathcal G \subseteq \mathcal O$ is that if \textit{there is a MOTS in the initial data $(h, \mathcal K) \in \mathcal G_\mathcal O,$ then  this MOTS is non-degenerate}. However our method cannot guarantee that such a MOTS exists for initial data in $\mathcal G.$  In order to obtain this last condition, we have to make extra assumptions about the fiducial MOTS triple $(h_0, \mathcal K_0, \psi_0).$ We cite two distinct situations where this is possible:
		\begin{enumerate}%[label = (\arabic*)]
			\item Assume that $(h_0, \mathcal K_0, \psi_0)$ is itself non-degenerate. Then by condition $T$ the partial derivative with respect to the embeddings variable at  $(h_0, \mathcal K_0, \psi_0)$ is an isomorphism (this can be more clearly seen from the abstract case in \cref{teo:banachgen}). Then by a straightforward application of the implicit function theorem in Banach spaces (cf. \textcite{abraham2012manifolds}, thm. 2.5.7) We can reduce $\mathcal A_0$ to a open set of the form $\mathcal U_0 \times \mathcal V_0,$ $(h_0,\mathcal K_0) \in \mathcal U_0,$ $\psi_0 \in \mathcal V_0,$ where now all initial data in $\mathcal O = U_0$ have a non-degenerate MOTS, so when the fiducial MOTS triple is non-degenerate we have obtained that existence of non-degenerate MOTS for initial data in $\mathcal U_0$ is \textit{stable}. (Actually, this argument bypasses the generic step (iii) in \cref{teo:banachgen} thanks to condition $T$ and the implicit function theorem.)
			\item For this second case, we restrict $(h_0, \mathcal K_0)$ and the dimension and topology of $\mathcal S$ to the following theorem due to   \textcite{andersson2011jang}: 
			\begin{theorem}[\cite{andersson2011jang}, thm. 3.3]\label{eichtheo}
				Let $(h, \mathcal K)$ be a initial data set for $\mathcal S,$ where $3\leq\dim \mathcal S \leq 7.$ Assume that there is a connected bounded open set $\Omega \subseteq \mathcal S$ with smooth embedded boundary $\partial \Omega.$ Assume this boundary consists of two non-empty closed hypersurfaces $\partial_+ \Omega$ and $\partial_- \Omega,$ possibly consisting of several components, so that
				\begin{equation}\label{eichmaircond}
					H_{\partial_+ \Omega} - \trace_{\partial_+ \Omega} \mathcal K > 0 \quad and \quad H_{\partial_- \Omega} + \trace_{\partial_- \Omega} \mathcal K > 0,
				\end{equation}
				where the mean curvature scalar is computed as the tangential divergence of the unit normal vector field that is pointing out of $\Omega$. Then there exists a smooth closed embedded MOTS\footnote{Other properties of $\Sigma$ are obtained in \cite{andersson2011jang}, but we need not concerned about them here.} $\Sigma \subseteq \Omega$ for the initial data $(h, \mathcal K)$ that is  homologous to $\partial_+ \Omega.$ \hfill $\ensuremath{\Box}$ 
			\end{theorem} 
			By restricting the topology and dimension of $\mathcal S$ as in \cref{eichtheo}, consider an initial data $(h_0, \mathcal K_0)$ for $\mathcal S,$ and then a MOTS $\psi_0$ for this initial data is obtained by the latter theorem. This MOTS triple $(h_0, \mathcal K_0, \psi_0)$ is now our fiducial MOTS triple, that may or may not be non-degenerated.
			The point is that conditions (\ref{eichmaircond}) are open, so there is an  open set $\mathcal U_0$ around $(h_0, \mathcal K_0)$ where \textit{all initial data in $\mathcal U_0$ have a MOTS} by \cref{eichtheo}. Applying \cref{teo:mainMOTS1} under these conditions, the residual set $\mathcal G$ will have the property that \textit{all initial data $(h, \mathcal K)$ in $\mathcal G$ have a non-degenerate MOTS.}
		\end{enumerate}
	\end{remark}

	Now let us organize \cref{teo:mainMOTS1} together with conditions laid out in \cref{rmk:veryimportant} as to obtain a generic condition for incompleteness.
	
	Consider a MOTS triple $(h_0, \mathcal K_0, \psi_0)$ and its open neighborhood $\mathcal A_0$ where $\theta_+$ is defined, assume now that $\psi_0(\Sigma)$ separates $\mathcal S$ as defined \cref{sec:motssing}. Reducing $\mathcal A_0$ to a smaller open set we can assume that all embeddings separate $\mathcal S.$ By \cref{teo:mainMOTS1} there is an open set $\mathcal O$ around $(h_0, \mathcal K_0)$ where the property (\ref{eq:motsgenericset}) is generic in $\mathcal O.$ If $(h, \mathcal K) \in \mathcal G$ \textit{has a MOTS}, this MOTS is non-degenerate, therefore we are in conditions to  apply  \cref{teo:motssingularity}, and MOTS for $(h,\mathcal K)$ can be deformed to be outer-trapped, then the maximal Cauchy development of $(h,\mathcal K)$ (cf. \cref{theo:CB_52_69}) is null incomplete. 
	
	Following \cref{rmk:veryimportant}, by suitably reducing $\mathcal A_0,$ under condition (1) null {incompleteness} is now \textit{stable} around $(h_0, \mathcal K_0),$ and under condition (2) it is \textit{generic} in an open subset around $(h_0, \mathcal K_0).$ We summarize this in the following theorem:

	\begin{theorem}\label{coro:motsgencoro}
		Let $(h_0, \mathcal K_0)$ be  initial data set,% satisfying the no KIDs condition,
		$\psi_0 : \Sigma \to \mathcal S$ an embedding that is a MOTS for the initial data $(h_0, \mathcal K_0),$ where $\psi_0(\Sigma)$ separates $\mathcal S.$ Consider the open set $\mathcal A_0$ around $(h_0, \mathcal K_0, \psi_0)$ where $\theta_+$ is defined and  all embeddings separates $\mathcal S,$ and also assume that $\theta_+$ satisfies the $T$ condition. Under these assumptions, 
		\begin{enumerate}%[label = {\textnormal{(\arabic*)}}]
			\item if $(h_0, \mathcal K_0, \psi_0)$ is itself non-degenerate, then there is an open set $\mathcal O$ around $(h_0, \mathcal K_0)$ where all initial data $(h,\mathcal K) \in \mathcal O$ are such that their maximal Cauchy development of $(h, \mathcal K)$ is null incomplete ({null incompleteness is \emph{stable}});
			\item if $(h_0, \mathcal K_0, \psi_0)$ is a MOTS triple originating from \cref{eichtheo} (where the topology and dimensions of $\mathcal S$ have to be further restricted)  then there is an open set $\mathcal O$ around $(h_0, \mathcal K_0)$ and a residual set $\mathcal G\subseteq O$ where all initial data $(h,\mathcal K) \in \mathcal G$  are such that their maximal Cauchy development of $(h, \mathcal K)$ is null incomplete ({null incompleteness is \emph{generic}}). \hfill $\ensuremath{\Box}$ 
		\end{enumerate}
	\end{theorem}

	\section*{Acknowledgements}
	IPCS is partially supported by the project PID2020-118452GBI00 of the Spanish government. He warmly thanks Paolo Piccione for discussions and early insights on the second part of this paper. VLE's work has been partially supported by the Coordenação de Aperfeiçoamento de Pessoal de Nível Superior – Brasil (CAPES) – Finance Code 001. 
	
	\noindent\textbf{Conflict of interest statement:} On behalf of all authors, the corresponding author states that there is no conflict of interest.
	
	\noindent \textbf{Data availability statement:} Data sharing not applicable to this article as no datasets were generated or analysed during the current study.

	\printbibliography
	%\bibliographystyle{plainnat}
	
	%\bibliography{ref.bib}
	
	% We choose the "plain" reference style

	%\bibliography{ref}

\end{document}